\documentclass[11pt, a4paper]{article}

\usepackage{subfigure}
\usepackage{graphicx}
\usepackage{natbib}
\usepackage{bm}
\usepackage{stmaryrd}
\usepackage{enumerate}
\usepackage[colorlinks,citecolor=blue,linkcolor=red,urlcolor=blue]{hyperref} 
\usepackage{times}

\usepackage{color}
\usepackage{lineno}
\usepackage{graphicx}
\usepackage{ae}
\usepackage{amsmath}
\usepackage{amssymb}
\usepackage{latexsym}
\usepackage{url}
\usepackage{epsfig}
\usepackage{mathrsfs}
\usepackage{amsfonts}
\usepackage{amsthm}
\usepackage{stmaryrd}
\usepackage{cite}
\usepackage{BOONDOX-cal}
\usepackage{indentfirst}
\usepackage{amssymb}
\usepackage{indentfirst}
\usepackage{algorithm}
\usepackage{algorithmic}
\numberwithin{equation}{section}
\allowdisplaybreaks

\def\qed{\hfill$\Box$\vspace{12pt}}

\usepackage{caption}
\usepackage{booktabs}
\usepackage{longtable}
\usepackage{multirow}
\usepackage{color}

\usepackage [T1]{fontenc}
\usepackage [utf8]{inputenc}
\usepackage{authblk}

\def\d{\mathrm{d}}

\newcommand{\var}{\mathrm{Var}}

\newcommand{\R}{\mathbb{R}}

\renewcommand{\[}{\left[}

\renewcommand{\ge}{\geqslant}

\renewcommand{\geq}{\geqslant}

\renewcommand{\epsilon}{\varepsilon}
\renewcommand{\cite}{\citet}

\DeclareMathOperator*{\argmax}{arg\,max}
\usepackage[misc]{ifsym}

\theoremstyle{plain}
\newtheorem{theorem}{Theorem}
\newtheorem{corollary}{Corollary}
\newtheorem{lemma}{Lemma}
\newtheorem{proposition}{Proposition}

\theoremstyle{definition}
\newtheorem{definition}{Definition}
\newtheorem{example}{Example}

\theoremstyle{remark}
\newtheorem{remark}{Remark}

\theoremstyle{definition}

\title{A non-zero-sum game  with  reinforcement learning under mean-variance framework }

\author[a]{Junyi Guo}
\author[a]{Xia Han}
\author[b]{Hao Wang\thanks{Corresponding author.

\ \ \text{E-mail addresses: jyguo@nankai.edu.cn (J. Guo); xiahan@nankai.edu.cn (X. Han);}

\ \ \text{hao.wang@mail.nankai.edu.cn (H. Wang); kcyuen@hku.hk(K.C. Yuen)}}}
\author[c]{Kam Chuen Yuen}
\affil [a] {School of Mathematical Sciences and LPMC, Nankai University, Tianjin, 300071, China}
\affil [b] {School of Mathematical Sciences, Nankai University, Tianjin, 300071, China}
\affil [c] {Department of Statistics and Actuarial Science, The University of Hong Kong, Pokfulam Road, Hong Kong}

\date{}

\begin{document}

\maketitle

\begin{abstract}

In this paper, we investigate a competitive market involving two agents who consider both their own wealth  and the wealth gap with their opponent. Both agents can invest in a financial market consisting of a risk-free asset and a risky asset, under conditions where model parameters are  partially or completely unknown. This setup gives rise to a  non-zero-sum differential game within the framework of reinforcement learning (RL). Each agent aims to maximize his own Choquet-regularized, time-inconsistent mean-variance objective. 
Adopting the dynamic programming approach, we  derive a time-consistent Nash equilibrium strategy in a general incomplete market setting. Under the additional assumption of  a Gaussian mean return model, we obtain an explicit analytical solution, which facilitates the development of a practical   RL algorithm. Notably, the proposed  algorithm achieves uniform convergence, even though  the conventional policy improvement theorem does not apply to  the equilibrium policy. Numerical experiments demonstrate the robustness and effectiveness of the algorithm, underscoring its potential for practical implementation.

\vspace{5mm}

\bigskip

\end{abstract}

\section{Introduction}

Since the seminal work of \cite{M52}, the mean-variance  criterion has emerged as a cornerstone in mathematical finance,  highlighted by significant contributions such as \cite{L00} and \cite{ZL00}.   It is well known that mean-variance problem has an inherent issue of time inconsistency,  thereby Bellman optimality principle cannot be applied. Many studies  addressing this issue resort to a pre-commitment strategy at the initial time and employ the Lagrangian method to  solve the mean-variance problem as seen in the aforementioned works. However, the pre-committed solution lacks time consistency and is applicable only for a decision-maker at the initial time.  
 
  In pursuit of a   time-consistent strategy, \cite{BM10} formulated the problem within a game theoretic framework and derived the  extended Hamilton-Jacobi-Bellman (HJB) equation  through a  verification theorem. Based on this groundwork,  \cite{BMZ14} assumed that the insurer's risk aversion is inversely proportional to the current wealth, and obtained the time-consistent strategy.   \cite{BKM17} provided a general framework for handling time-inconsistency, leading to the so-called equilibrium policy that can be regarded as a subgame perfect Nash equilibrium in dynamic games. Expanding on these efforts,  \cite{DJKX21}  proposed a dynamic portfolio choice model with the mean-variance criterion for log-returns, and derived
 time-consistent portfolio policies which are analytically tractable even under some incomplete market
settings. Furthermore, the game theoretic approach to tackling time-inconsistent problems has also been extended in the actuarial literature, as demonstrated by works from  \cite{ZLG16}, \cite{CS19}, and \cite{LY21}.

In this paper, we consider the scenario where model parameters are partially or completely unknown, aiming to learn a practical exploratory equilibrium policy under the mean-variance criterion within an incomplete market, as discussed in \cite{DDJ23}.  
In fact,  the stochastic control problems  under the continuous-time RL framework with continuous state and action  have attracted extensive attention from scholars recently. \cite{WZZ20} first established a continuous-time  RL framework with continuous state and action from the perspective of stochastic control and proved that the optimal exploration strategy for the linear–quadratic (LQ) control problem in the infinite time horizon is Gaussian. Furthermore, \cite{WZ20} applied this RL framework for the first time to solve the continuous-time mean-variance  problem. 
Motivated by   \cite{WZZ20},  \cite{DDJ23} extended the exploratory stochastic control framework  to an incomplete
market, where the asset return correlates with a stochastic market state,  and learned  an equilibrium policy under  a mean-variance criterion. \cite{JSW22}  studied
the exploratory Kelly problem by considering both the amount
of investment in stock and the portion of wealth in stock as
the control for a general time-varying temperature
parameter.  \cite{HWZ23} first  introduced another kind of index that can measure the randomness of actions called Choquet regularization. They showed that  the optimal exploration distribution of LQ control problem with infinite time horizon is  no longer necessarily Gaussian as in \cite{WZZ20}, but are dictated by the choice of Choquet regularizers. \cite{GHW23} further  studied    an exploratory mean-variance   problem  with the Choquet regularizers being used to measure the level of exploration.

Interactions among agents in real-world scenarios often entail a blend of competitive and cooperative dynamics. Thus, in contrast to the aforementioned references, this paper assumes two competitive agents who consider not only their own wealth but also the wealth gap with their opponent. This setup aligns with prior investigations into non-zero-sum stochastic differential games, which can be traced back to seminal works by \cite{I65} and \cite{P67}, and subsequently expanded upon by scholars including \cite{E76}, \cite{BF00},  \cite{B20}, \cite{ES11},  \cite{BSYY14}, and \cite{ET15}. To the best of our knowledge,   (non-)zero-sum games  in the continuous-time RL settings have not been considered before except in \cite{SJ23}  where  an entropy-regularized continuous-time LQ two-agent zero-sum stochastic differential game problem was considered,  and  they  designed a policy iteration method to derive the optimal strategy for a case with only one unknown model parameter.

Compared with the existing literature, this paper presents three main differences and contributions.

\textbf{Firstly}, in  traditional time-consistent optimization problems, policy iteration typically relies on the policy improvement theorem, as detialed in \cite{JZ22b} and \cite{GHW23}. This theorem ensures that each iteration enhances the overall strategy. However, for  time-inconsistent problems,  although iterating policies is still  feasible, there is no guarantee  that each iteration  leads to an improved policy.    
This presents a significant challenge when attempting to extend policy iteration methods to time-inconsistent settings. Nevertheless, we show that, our approach guarantees uniform convergence of the equilibrium strategy, in contrast to the local convergence typically observed in the single-agent framework, as shown in \cite{DDJ23}.

\textbf{Secondly}, unlike single-agent scenarios, our non-zero-sum game problem naturally fits within the domain of applying RL techniques in multi-agent systems. The pioneering work of \cite{L94} introduces  Q-learning in zero-sum games, marking the early development of multi-agent reinforcement learning. Building on this , \cite{Lit01} proposed the Friend-or-foe Q-learning algorithm for general-sum games, while \cite{FNFATKW17} proposed a centralized  multi-agent learning method using deep learning, which enhances the collaborative to improve agent collaboration in complex tasks. Although centralized algorithms theoretically guarantee convergence and stability, they often encounter practical challenges, including dimensionality explosion and increased system complexity. For further discussion on multi-agent algorithms,   we refer to \cite{YW20} and \cite{ZYB21}. In contrast, the unique structure of our model allows the differential stochastic game to be decomposed into two independent single-agent problems within a centralized multi-agent framework, thus mitigating the dimensionality explosion inherent in centralized methods. To the best of our knowledge, this paper is the first to address equilibrium policies in time-inconsistent problems within the context of reinforcement learning for non-zero-sum differential games.

\textbf{Thirdly}, in contrast to \cite{JSW22}, \cite{DDJ23}, and \cite{SJ23}, we replace the differential entropy used for regularization with Choquet regularizers. As noted in \cite{HWZ23} and \cite{GHW23}, Choquet regularizers offer several theoretical and practical advantages RL. The broad class of Choquet regularizers  enable the comparison and selection of specific regularizers to meet the unique objectives of each learning problem. In particular, it is more natural for agents to choose different regularizers based on their individual preferences, further enhancing the flexibility and applicability of our approach.

The rest of the paper is organized as follows. In Section \ref{sec:2}, we introduce the exploratory mean–variance problem within the framework of RL  under the non-zero-sum differential game setting. Section \ref{sec:3} presents   the  Nash equilibrium mean-variance policy. In Section \ref{sec:4}, we show a policy iteration procedure and analyze its convergence  based on the Gaussian mean return model. In Section \ref{sec:5}, we propose  an RL algorithm based on the convergence analysis,  and provide numerical results to illustrate the implementation of the algorithm in Section \ref{num}.   Finally, we conclude in Section \ref{sec:6}.

\section{Formulation of problem}\label{sec:2}
Throughout the paper, we assume that $(\Omega,\mathcal F, \mathbb P)$ be an atomless probability space. With a slight abuse of notation, we denote by $\mathcal M$  both the set of Borel probability measures on $\mathbb R$ and the set of distribution functions of real random variables. For $\Pi \in \mathcal M$ and $x \in \mathbb R$, we have $\Pi(x)=\Pi((-\infty , x])$. Let $\mathcal M ^p,\ p \in [1,\infty)$, be the subset of $\mathcal M$ whose elements have finite $p$-th moment. We write $X \sim \Pi$ if random variable $X$ has distribution $\Pi$, and $X \overset{d}= Y$ if $X$ and $Y$ have the same distribution. For $\Pi \in \mathcal M^2$, we denote by $\mu(\Pi)$ and $\sigma ^2 (\Pi)$ the mean and variance of $\Pi$, respectively. 
\subsection{Exploratory Wealth Process}
In the financial market under study,  we posit the inherent incompleteness of the market as in \cite{BC10} and \cite{DJKX21}, allowing for continuous trading in both a risk-free asset and a risky asset within the finite time horizon $[0,T]$. The price process $S_0(t)$  governing the risk-free asset is  given by 
\begin{align*}
	\d S_0(t) = rS_0(t)\d t,
\end{align*}
where $r>0$ represents the constant risk-free interest rate. The price process $S(t)$ of the risky asset is given by
\begin{align}\label{dynamics}
	\dfrac{\d S(t)}{S(t)}=a(t,Y(t))\d t+b(t,Y(t))\d B(t).
\end{align}
Here, $B(t)$ is a Brownian motion defined on the filtered probability space $(\Omega, \mathcal F, \{\mathcal F_t\}_{0\leqslant t\leqslant T}, \mathbb P)$ adhering to standard conditions. Additionally,  $Y(t)$ is a diffusion process satisfying
\begin{align*}
	Y(t)=m(t,Y(t))\d t+v(t,Y(t))[\rho dB(t)+\sqrt{1-\rho ^2}d\widetilde B(t)],
\end{align*}
where $\rho \in [0,1]$ and the  Brownian motion $\widetilde B(t)$ is  defined on $(\Omega, \mathcal F, \{\mathcal F_t\}_{0\leqslant t\leqslant T}, \mathbb P)$ independent of $B(t)$.

Within the market context,  the functions $a(\cdot , \cdot)$, $b(\cdot , \cdot)$, $m(\cdot , \cdot)$ and $v(\cdot , \cdot)$ are deterministic in nature. The state process $Y(t)$ can represent some economic factors associated with the expectation or volatility of stock return.  Examples include dividend yields in predictability models or instantaneous volatility in stochastic volatility models. We note that the market is incomplete as
trading in the stock and bond cannot perfectly hedge the changes in the stochastic investment
opportunity set. The market model described above exhibits substantial generality, encompassing numerous well-known models such as the time-varying Gaussian mean return model and the stochastic volatility model   as special  cases; see Example \ref{exa:1} below.  Particularly, we are in the case of the complete market with constant parameters, where $a(t, y) \equiv a$ and $b(t, y) \equiv b$.
\begin{example}\label{exa:1}
\begin{itemize}\item[(i)]If  the stock price $S(t)$ and the market price of risk $Y(t)$ are governed by
$
\d S(t)/{S(t)}  =(r+\sigma Y(t)) \d t+\sigma \d B(t)$
with 
$
\d Y(t)  =\iota \left(Y-Y(t)\right) \d t+v \d \tilde B(t),
$
where $r, \sigma,  \iota, v$, and $Y$ are all positive constants.   This model, termed the time-varying Gaussian mean return model, effectively captures the intricate interplay between the dynamics of stock prices and market risk over time.
\item[(ii)] If the stock price $S(t)$ and a state variable $Y(t)$ follow
$
\d S(t)/S(t)=(r+\sigma Y(t)^{\frac{1+\alpha}{2 \alpha}}) \d t+Y(t)^{\frac{1}{2 \alpha}} \d B(t)
$
with 
$
\d Y(t)=\iota (Y-Y(t)) d t+v \sqrt{Y(t)} \d \tilde B(t),
$
where $\alpha \neq 0$ is the constant elasticity of the market price of risk $\sigma Y(t)^{\frac{1}{2}}, \sigma \in \mathbb{R}, \iota>0, v>0$, and $Y \in \mathbb{R}$ are all constants, then it is  a stochastic volatility model. This formulation characterizes a stochastic volatility model, acknowledging the non-linear relationship between the state variable and market risk elasticity. 
\end{itemize}
Researchers, including \cite{M80}, \cite{KO96}, \cite{L01}, \cite{W02}, \cite{BC10}, \cite{DJKX21}, and \cite{DDJ23}, have extensively explored dynamic portfolio choice problems within the framework of these two market settings and their specific variations.
\end{example}
In what follows, we consider a non-zero-sum game involving two competing agents or companies, referred  as Agent 1 and Agent 2 for simplicity. Both agents have access to the financial market.
 For $i \in \{1,2\}$, let $u_i=\{u_i(t),0\leqslant t\leqslant T\}$ represent the discounted amount of Agent $i$ invested in the risky asset at time $t$, and the rest of the wealth is invested in the risk-free asset. Define \begin{equation}\label{eq:theta}\theta(t,Y(t))=\frac{a(t,Y(t))-r}{b(t,Y(t))}.\end{equation} Then the  dynamic of discounted wealth process of Agent $i$ under strategy $u_i$ is given as
\begin{align}\label{dynamicx}
	\d X_i^{u_i}(t)&=u_i(t)b(t,Y(t))[\theta(t,Y(t))\d t+\d B(t)].
\end{align}

Based on \cite{WZZ20}, we extend   the control process  \eqref{dynamicx} to a distributional control process and define  the exploratory discounted wealth process for agents. 
Denote by  $\Pi_i(t) \in \mathcal M (U)$  the probability distribution function of  the control $u_i$ at time $t$, where $\mathcal M (U)$ represents the set of distribution functions on $U$.  We write $\Pi_i \in \mathcal M$  for simplicity. It is assumed that the actions of the agent are independent of both   $B(t)$ and $\widetilde B(t)$.  In contrast to the approach presented in \cite{WZZ20}, we additionally consider the correlation between the discounted wealth process and the market state, as well as the correlation between the two agents.  

Now, we attempt to derive the exploratory version of the wealth process  $X_i^{u_i}$ associated with randomized policy $\Pi_i$.  Let $B^n(t)$ and $\widetilde B^n(t),n=1,2,...,N,$ represent $N$ paths independently sampled from $B(t)$ and $\widetilde B(t)$, respectively. Moreover, let $X_i^n(t)$ be the copies of the discounted wealth process of Agent $i$ under strategy $u_i^n$ sampled from $\Pi_i$. Then, for $n=1,2,...,N$ and $i \in \{1,2\}$, the increments of $Y^n(t)$ and the corresponding $X_i^n(t)$ can be written as
\begin{align}\label{delta_Y}
\begin{aligned}
	\Delta Y^n(t)&\equiv Y^n(t+\Delta t)-Y^n(t)\\&\approx m(t,Y^n(t))\Delta t+v(t,Y^n(t))[\rho \Delta B^n(t)+\sqrt{1-\rho ^2}\Delta \widetilde B^n(t)],\end{aligned}
\end{align} and\begin{align}\label{delta_X}
\begin{aligned}
	\Delta X_i^n(t)&\equiv X_i^n(t+\Delta t)-X_i^n(t)\\&\approx u_i^n(t)b(t,Y^n(t))[\theta(t,Y^n(t))\Delta t+\Delta B^n(t)].
\end{aligned}
\end{align}
We denote the exploratory discounted wealth process of Agent $i$ by $X_i^{\Pi_i}(t)$. Consequently, $X_i^n(t)$ can be veiwed as an independent sample from $X_i^{\Pi_i}(t)$. By the law of large numbers and using \eqref{delta_Y}  and \eqref{delta_X},  we observe that, as $N \to \infty$,
\begin{align}
\begin{split}\label{firstmoment}
	&\dfrac{1}{N}\sum\limits_{n=1}^N \Delta X_i^n(t)\approx \dfrac{1}{N}\sum\limits_{n=1}^N u_i^n(t)b(t,Y^n(t))[\theta(t,Y^n(t))\Delta t+\Delta B^n(t)]\\
	&\overset{a.s.}\longrightarrow \mathbb E\left[b(t,Y(t))\theta(t,Y(t))\int_U ud\Pi_i(t,u)\right]\Delta t,
\end{split}
\\
\begin{split}\label{secondmoment}
	&\dfrac{1}{N}\sum\limits_{n=1}^N (\Delta X_i^n(t))^2\approx \dfrac{1}{N}\sum\limits_{n=1}^N (u_i^n(t)b(t,Y^n(t)))^2\Delta t \\
	&\overset{a.s.}\longrightarrow \mathbb E\left[b^2(t,Y(t))\int_U u^2d\Pi_i(t,u)\right]\Delta t,
\end{split}
\end{align}
and 
\begin{align}
\begin{split}\label{bracket}
	&\dfrac{1}{N}\sum\limits_{n=1}^N \Delta X_i^n(t)\Delta Y^n(t)\approx \dfrac{1}{N}\sum\limits_{n=1}^N u_i^n(t)b(t,Y^n(t))\rho v(t,Y^n(t))\Delta t\\
	&\overset{a.s.}\longrightarrow \mathbb E\left[\rho b(t,Y(t))v(t,Y(t))\int_U ud\Pi_i(t,u)\right]\Delta t.
\end{split}
\end{align}
It is well known from the law of large numbers that
$$
	\dfrac{1}{N}\sum\limits_{n=1}^N \Delta X_i^n(t)\overset{a.s.}\longrightarrow \mathbb E[\Delta X_i^{\Pi_i}(t)],$$  $$\dfrac{1}{N}\sum\limits_{n=1}^N (\Delta X_i^n(t))^2\overset{a.s.}\longrightarrow \mathbb E[(\Delta X_i^{\Pi_i}(t))^2],
$$
and
\begin{align*}
	\dfrac{1}{N}\sum\limits_{n=1}^N \Delta X_i^n(t)\Delta Y^n(t)\overset{a.s.}\longrightarrow \mathbb E[\Delta X_i^{\Pi_i}(t)\Delta Y(t)].
\end{align*}
These together with \eqref{firstmoment} -- \eqref{bracket}, motivate our confidence in  the dynamic of $\Delta X_i^{\Pi_i}(t)$, which can be expressed as 
\begin{align}\label{agentdynamic}
\begin{split}
	\d X_i^{\Pi_i}(t)&=b(t,Y(t))\theta(t,Y(t))\mu_i(t)\d t+b(t,Y(t))[\mu_i(t)\d B(t)+\sigma_i(t)\d \overline B_i(t)],
\end{split}
\end{align}
where $$\mu_i(t)=\int_U ud\Pi_i(t,u), ~~\sigma^2_i(t)=\int_U u^2d\Pi_i(t,u)-\mu^2_i(t)$$ with  $\overline B_i$ being the Brownian motion  independent of $B(t)$ and $\widetilde B(t)$. The remaining consideration involves the correlation between  $\overline B_1(t)$ and $\overline B_2(t)$.  It is observed that, as $N\to \infty$,
\begin{align*}
	&\dfrac{1}{N}\sum\limits_{n=1}^N \Delta X_1^n(t)\Delta X_2^n(t)\approx \dfrac{1}{N}\sum\limits_{n=1}^N u_1^n(t)u_2^n(t)b^2(t,Y^n(t))\\
	&\overset{a.s.}\longrightarrow \mathbb E\left[\mu_1(t)\mu_2(t)b^2(t,Y(t))\right]\Delta t,
\end{align*}
leading to $\langle  \overline B_1,\overline B_2  \rangle =0$. By L\'evy's theorem,  we conclude  that $\overline B_1$ is independent of $\overline B_2$.

For $i,j\in\{1,2\}$ and $j\neq i$,  assume that Agent $i$  takes into account not only his own wealth but also the wealth gap between himself and Agent $j$ at the terminal time $T$.  That is, given a strategy $\Pi_j\in \mathcal M$  employed by Agent $j$,   Agent $i$ will choose a strategy $\Pi_i$ to maximize the following objective
\begin{align}\label{mvcriterion}
\begin{split}
	&~\mathbb E_t[(1-k_i)X_i^{\Pi_i}(T)+k_i(X_i^{\Pi_i}(T)-X_j^{\Pi_j}(T))]-\dfrac{\gamma_i}{2}\var_t[(1-k_i)X_i^{\Pi_i}(T)+k_i(X_i^{\Pi_i}(T)-X_j^{\Pi_j}(T))]\\
	&=\mathbb E_t[X_i^{\Pi_i}(T)-k_iX_j^{\Pi_j}(T)]-\dfrac{\gamma_i}{2}\var_t[X_i^{\Pi_i}(T)-k_iX_j^{\Pi_j}(T)],
\end{split}
\end{align}
 where $\mathbb E_t(\cdot)$ and $\var_t(\cdot)$ denote the conditional expectation and variance with given $X_1^{\Pi_1}(t)$, $X_2^{\Pi_2}(t)$ and $Y(t)$, respectively,  $\gamma_i>0$ represents the risk-aversion coefficient for Agent $i$, and $k_i\in (0,1)$ measures the sensitivity of Agent $i$ to the performance of Agent $j$. Notably, a larger  $k_i$
 implies that  Agent $i$ emphasis on the relative performance against his opponent (Agent 
$j$), thereby intensifying the competitiveness of the game. 

Let $\hat X_i^{\Pi_i,\Pi_j}(t)=X_i^{\Pi_i}(t)-k_iX_j^{\Pi_j}(t)$ be the wealth difference of the two agents. It is obvious  from \eqref{agentdynamic} that $\hat X_i^{\Pi_i,\Pi_j}(t)$ follows the dynamic
\begin{align}\label{dynamic}
\begin{split}
	\d \hat X_i^{\Pi_i,\Pi_j}(t)=&b(t,Y(t))[\theta(t,Y(t))(\mu_i(t)-k_i\mu_j(t))\d t\\
	&+(\mu_i(t)-k_i\mu_j(t))\d B(t)+\sigma_i(t)\d\overline B_i(t)-k_i\sigma_j(t)\d\overline B_j(t)].
\end{split}
\end{align}
\subsection{Objective function}

We employ  the \emph{Choquet regularizer} $\Phi_h$ to quantify  randomness.   Given a concave function $h:[0,1]\to \R$  of bounded variation with $h(0)=h(1)=0$ and   $\Pi\in \mathcal M$,
  the Choquet regularizer $\Phi_h$  on $\mathcal M$ is defined as 
 \begin{equation}\label{eq:phi_h}
\Phi_h(\Pi)= \int_\R h\circ \Pi([x,\infty))\d x.\end{equation}
We denote the set of $h:[0,1]\to \R$  by   $\mathcal H$.   In fact, \eqref{eq:phi_h}    is a signed Choquet integral characterized by  \cite{WWW20c}  via comonotonic additivity, which essentially builds on the seminal works of \cite{S89} and \cite{Y87}.
 
As stated in Lemma 2.2 of  \cite{HWZ23}, the regularizer $\Phi_h$ is rigorously defined and serves as an important  metric for quantifying the degree of randomness or exploration within the context of RL.   Specifically, the concavity of 
$h$  ensures that $\Phi_h$   is also  concave. This means  that 
$\Phi_h(\lambda \Pi_1 + (1-\lambda) \Pi_2 ) \ge \lambda \Phi_h( \Pi_1)  + (1-\lambda) \Phi_h(\Pi_2 )$ for all $\Pi_1,\Pi_2\in \mathcal M$ and $\lambda\in [0,1],$    which intuitively implies that the linear combination of two distributions is more random. 
Moreover,  the condition $h(0)=h(1)=0$ means for   any $c\in\R$, $\Phi_h(\delta_c)=0$, where $\delta_c$ is the Dirac mass at $c$. This indicates that  degenerate distributions do not have any randomness measured by $\Phi_h$.
Additionally,   the agents have the flexibility to opt for different regularizers, contingent upon their preferences, as reflected by the distortion function $h$. 	For  more detailed discussions about the properties associated with $\Phi_h$, we refer to \cite{HWZ23} and \cite{GHW23}. 
 
 It is useful to note that $\Phi_h$ admits a quantile representation, see Lemma 1 of \cite{WWW20b}. For a distribution $\Pi\in\mathcal{M}$,  let its   left-quantile for $p\in(0,1]$ be defined as $$Q_\Pi(p)=\inf \left\{x\in \R: \Pi(x) \ge p\right\}, $$
 then we have  \begin{equation}\label{eq:quan_rep}\Phi_h(\Pi) =\int_0^1  Q_\Pi(1-p) \d  h(p) \end{equation} if $h$ is left-continuous.

 For any fixed $\Pi_j \in \mathcal M$, we  incorporate the Choquet regularizer $\Phi_{h_i}$ for Agent $i$,  along with the exploration weight function $\lambda_i(t)$,  into the mean-variance criterion \eqref{mvcriterion}.   Each agent aims to achieve an exploratory mean-variance problem within the framework of RL.   This yields the corresponding objective function
\begin{align}\label{eq:obj}
	J_i(t,\hat x_i,y;\Pi_i,\Pi_j):=\mathbb E_t\left[\int_t^T\lambda_i(s)\Phi_{h_i}(\Pi_i(s))ds+\hat X_i^{\Pi_i,\Pi_j}(T)\right]-\dfrac{\gamma_i}{2}\var_t[\hat X_i^{\Pi_i,\Pi_j}(T)],
\end{align}
where   $\hat X_i^{\Pi_i,\Pi_j}(t)=\hat x_i$ and $\ Y(t)=y$ with $t$ representing the initial time.
\begin{remark} 
We note that a larger $\lambda_i(t)$ promotes increased exploration, as it results in a higher weight  on  $\Phi_{h_i}(\Pi_i(t))$. When $\lambda_i(t) \equiv \lambda_i$, the  exploration weight  remains constant over time. It is often more realistic to set $\lambda_i(t)$ to decrease over time;  for instance, $\lambda(t)$ may follow a power-decaying pattern, expressed as   $\lambda_i(t) = \lambda_0 (T+\lambda)^{\lambda_0}/(t+\lambda)^{\lambda_0+1} $ with  $\lambda_0, \lambda > 0$. Alternatively, $\lambda_i(t)$ may decay exponentially  with  $\lambda_i(t) = \lambda_0 e^{\lambda_0(T-t)}$ with $\lambda_0 > 0$.   For further discussions on selecting $\lambda_i(t)$, we refer to Section 3.4 of \cite{JSW22}. 
 \end{remark}
  For $i,j\in\{1,2\}$ and $i\neq j$,  denote by $\mathcal A(\Pi_j)$ the set of all admissible feedback control  of $\Pi_i$.  Given $\Pi_j \in \mathcal M$,  $\Pi_i$ is said to be in $\mathcal A(\Pi_j)$ if the following conditions hold: 
  \begin{itemize}
\item[(i)]
		For each $s\in[t,T]$, $\Pi_i(s)\in \mathcal M(\mathbb R)$;
		\item[(ii)] There exists a deterministic mapping $\pi_i:[t,T]\times\mathbb R \times\mathbb R\to \mathcal M(\mathbb R)$, such that $\Pi_i(s)=\pi_i(s,\hat X_i^{\Pi_i,\Pi_j}(s),Y(s))$;
	\item[(iii)] For any $A \in \mathcal B(\mathbb R)$, $\{\int_A\Pi_i(s,u)\d u, t\leqslant s \leqslant T\} $ is $\mathcal F_s$--progressively measurable;
		\item[(iv)] $\mathbb E_t\int_t^T [|b(s,Y(s))\theta(s,Y(s))(\mu_i(s)-k_i\mu_j(s))|+b^2(s,Y(s))(\mu_i^2(s)+\sigma^2_i(s))]ds< \infty$;
		\item[(v)] $\mathbb E_t\int_t^T|\lambda_i(s)\Phi_{h_i}(\Pi_i(s))|ds <\infty $.
\end{itemize}
Next, we define the \emph{profile} of the game, which   encapsulates the comprehensive strategic behavior of both agents in the game.

\begin{definition}\label{profile}
	A strategy pair $(\Pi_1,\Pi_2)$ is called the \emph{profile} of the game if it satisfies the following conditions:  given the strategy $\Pi_2$, $\Pi_1$ is an admissible feedback control, and   conversely,  given the strategy $\Pi_1$, $\Pi_2$ is also an admissible feedback control.\end{definition}

As mentioned in the introduction, to seek the time-consistent equilibrium strategy, we formulate the time-inconsistent dynamic optimization problem into a noncooperative game theoretic
framework proposed by  \cite{BM14} and \cite{BKM17}.      
\begin{definition}\label{equilibriumresponse}
	Let $(\Pi_1,\Pi_2)$ be a profile of the game. For $i,j\in\{1,2\}$ and $i\neq j$, $\Pi_i$ is said to be \emph{equilibrium response} of Agent $i$ if for a fixed $h$ and for any initial state $(t,\hat x_i,y)$ and an arbitrary $v_i\in\mathcal M$, $\Pi_i^{v_i,h}$ defined by
	\begin{align*}
		\Pi_i^{v_i,h}(s)=\left\{
		\begin{aligned}
			&v_i,&t\leqslant s\leqslant t+h,\\
			&\Pi_i(s),&t+h\leqslant s\leqslant T,
		\end{aligned}
		\right.
	\end{align*}
	satisfying
	\begin{align*}
		\liminf\limits_{h\rightarrow 0^+}\dfrac{J_i(t,\hat x_i,y;\Pi_i,\Pi_j)-J_i(t,\hat x_i,y;\Pi_i^{v_i,h},\Pi_j)}{h}\geqslant 0.
	\end{align*} The equilibrium response $\Pi_i$ of Agent $i$ can be viewed as a mapping of $\Pi_j$, and thus can be written as $\Pi_i=\pi_i(\Pi_j)$ at this point.   Furthermore, the equilibrium response value function of Agent $i$ is defined as
	\begin{align*}
		\widetilde V_i(t,\hat x_i,y;\Pi_j):=J_i(t,\hat x_i,y;\Pi_i,\Pi_j).
	\end{align*}
\end{definition}

\begin{definition}
	A profile $(\Pi_1^*,\Pi_2^*)$ is called the \emph{time-consistent Nash equilibrium} of the game if for $i,j\in\{1,2\}$ and $i\neq j$, $\Pi_i^*$ is the equilibrium response of Agent $i$, that is $\Pi_1^*=\pi_1(\Pi_2^*)$ and $\Pi_2^*=\pi_2(\Pi_1^*)$. Furthermore, the equilibrium value function of Agent $i$ is given by
	\begin{align*}
		V_i(t,\hat x_i,y):=\widetilde V_i(t,\hat x_i,y;\Pi^*_j)=J_i(t,\hat x_i,y;\Pi_i^*,\Pi_j^*).
	\end{align*}
\end{definition}
\section{Time-consistent Nash equlibrium}\label{sec:3}
In this section, we present the Nash equilibrium for a general incomplete market. It is important to emphasize that the uniqueness of the equilibrium policy remains an open question in the context of time-inconsistent optimization problems, as discussed by  \cite{EP08}.   Therefore, within the framework of game theory, we focus on a specific Nash equilibrium.
\subsection{Verification theorem}Analogous to the work of  \cite{BKM17} on optimization problems involving a broad class of objective functionals,    we present the following verification theorem.

Let  $\mathcal{D}=[0,T]\times\mathbb{R}^2$ and
for any $\psi\in C^{1,2,2}(\mathcal{D})$,  define the generator 
\begin{equation}\label{eq:oper}\begin{aligned} \mathcal L^{\Pi_i,\Pi_j}\psi(t, x, y)=&\displaystyle\frac{\partial\psi}{\partial t}+b(t,y)(\theta(t,y)(\mu_i(t)-k_i\mu_j(t))  \frac{\partial\psi}{\partial  x} +\frac{1}{2}b^2(t,y)\left((\mu_i(t)-k_i\mu_j(t))^2\right.\\&\left.+\sigma^2_i(t)+k^2_i\sigma^2_j(t)\right)\frac{\partial^2\psi}{\partial  x^2} + m(t,y)\frac{\partial\psi}{\partial y}+\frac{1}{2}v^2(t,y) \frac{\partial^2\psi}{\partial y^2}\\&+ \rho v(t,y) b(t,y)\theta(t,y)(\mu_i(t)-k_i \mu_j(t))\frac{\partial^2\psi}{\partial  x \partial y }.\end{aligned}\end{equation}
\begin{theorem}[Verification theorem]
	For $i,j\in\{1,2\}$ and $i\neq j$, fix $\Pi_j$ and suppose that functions $W_i(t,x,y)\in C^{1,2,2}(\mathcal{D})$, $g_i(t,x,y)\in C^{1,2,2}(\mathcal{D})$ and strategy $\Pi_i$ satisfy the following properties:
	\begin{enumerate}[i)]
		\item $W_i$ and $g_i$ solve the extended HJB system
		\begin{align}\label{hjbw}
			&\sup\limits_{\pi_i\in \mathcal M}\left\{\mathcal L^{\pi_i,\Pi_j}W_i(t,\hat x_i,y)-\dfrac{\gamma_i}{2}\mathcal L^{\pi_i,\Pi_j}g_i^2(t,\hat x_i,y)+\gamma_ig_i\mathcal L^{\pi_i,\Pi_j}g_i(t,\hat x_i,y)+\lambda_i(t)\Phi_{h_i}(\pi_i) \right\}=0,
		\end{align}
		and
		\begin{align}\label{hjbg}
			\mathcal L^{\pi_i,\Pi_j}g_i(t,\hat x_i,y)=0
		\end{align}
		with
		\begin{align}\label{terminalcondition}
			W_i(T,\hat x_i,y)=\hat x_i,\ g_i(T,x,y)=\hat x_i,
		\end{align}
		where $\mathcal L^{\pi_i,\Pi_j}$ is the infinitesimal generator given by \eqref{eq:oper}.
		\item $\pi_i$ realizes the supremum in \eqref{hjbw} and $\Pi_i(s)=\pi_i(s,\hat X_i^{\Pi_i,\Pi_j}(s),Y(s))$ is admissible.
	\end{enumerate}
	Then $\Pi_i$ is the equilibrium response of Agent $i$. Furthermore, $W_i(t,\hat x_i,y)=J_i(t,\hat  x_i,y;\Pi_i,\Pi_j)$ is the equilibrium response value function of Agent $i$ and $g_i(t,\hat x_i,y)=\mathbb E_t[\hat X_i^{\Pi_i,\Pi_j}(T)]$.
\end{theorem}

\subsection{Solution to the general case}
While the preservation of the uniqueness of equilibrium response remains uncertain, the aforementioned theorem provides a constructive framework for identifying a specific Nash equilibrium.  We first concentrate on the equilibrium response of Agent $1$, and analogous results for Agent $2$ can be obtained using the same method.  It is assumed that the equilibrium value function $V_1$ and the function $g_1$  can be precisely expressed as 
  \begin{align}
	V_1(t,x,y)=x+D_1(t,y),~~~g_1(t,x,y)=x+d_1(t,y).\label{g1}
\end{align}
It is obvious that $D_1(T,y)=d_1(T,y)=0$.  
Using \eqref{eq:oper} and substituting \eqref{g1} into \eqref{hjbw}, we streamline \eqref{hjbw} to \begin{align}\label{simplifyhjb}
\begin{split}
	\sup\limits_{\pi_1\in \mathcal M}&\left\{ \dfrac{\partial D_1(t,y)}{\partial t}+b(t,y)\theta(t,y)(\mu_1(t)-k_1\mu_2(t))+m(t,y)\dfrac{\partial D_1(t,y)}{\partial y}\right.\\
	&-\dfrac{\gamma_1}{2}b^2(t,y)[(\mu_1^2(t)+\sigma_1^2(t))+k_1^2(\mu_2^2(t)+\sigma_2^2(t))-2k_1\mu_1(t)\mu_2(t)]\\
	&-\dfrac{\gamma_1}{2}v^2(t,y)\left(\dfrac{\partial d_1(t,y)}{\partial y} \right)^2+\dfrac{1}{2}v^2(t,y)\dfrac{\partial ^2 D_1(t,y)}{\partial y^2}\\
	&\left.-\gamma_1\rho v(t,y)b(t,y)(\mu_1(t)-k_1\mu_2(t))\dfrac{\partial d_1(t,y)}{\partial y}+\lambda_1(t)\Phi_{h_1}(\pi_1)\right\}=0.
\end{split}
\end{align}
We can see from \eqref{simplifyhjb} that the supremum only depends on the mean and variance of $\pi_1$ except $\Phi_{h_1}(\pi_1)$. 
 We proceed with our analysis relying on the crucial Lemma \ref{lem:3.2}.
  \begin{lemma}[Theorem 3.1 of \cite{LCLW20}]\label{lem:3.2}
If $h$ is continuous and not constantly zero, then
a maximizer $\Pi^*$ to the optimization problem \begin{align}\label{eq:opt}
\max_{\Pi \in \mathcal M^2}\Phi_h (\Pi) \mbox{~~~~~subject to~} \mu (\Pi) =m ~\mbox{and}~ \sigma^2 (\Pi) = s^2
\end{align}  has the following quantile function
\begin{equation}\label{eq:lemma3}
Q_{\Pi^*}(p) =  m + s\frac{ h'(1-p) }{ ||h'||_2}, ~~ \mbox{~a.e. }p\in (0,1),
\end{equation}
and the  maximum value of \eqref{eq:opt} is $\Phi_h(\Pi^*)= s||h'||_2$. 
\end{lemma}

 Define the expression enclosed in brackets  of \eqref{simplifyhjb}  as  $\varphi_1(\pi_1)$. We have the following proposition.

\begin{proposition}\label{prop:3.3}  Let   $h_1\in \mathcal H$ be a continuous function, and  let   $\Pi_1\in\mathcal A(\Pi_2)$ be an admissible strategy.  
 For any  strategy  $\Pi_1=\{\Pi_1(t)\}_{t\ge 0}\in \mathcal A(\Pi_2)$ with mean process $\{\mu_1(t)\}_{t\ge 0}$ and variance process $\{\sigma_1(t)^2\}_{t\ge 0}$, there exists a strategy   $\Pi^*_1=\{\Pi^*_1(t)\}_{t\ge 0}\in \mathcal A(\Pi_2)$ defined  by
 \begin{equation*}
Q_{\Pi^*_1(t)}(p) = \mu_1(t) + \sigma_1(t)\frac{ h_1'(1-p) }{ ||h_1'||_2}, ~~ ~~\mbox{~a.e.}~p\in (0,1),\;\;t\geq0,
\end{equation*}
 which shares the same mean  and variance processes  as $\Pi_1$ and satisfies 
 $\varphi_1(\pi_1) \leqslant \varphi_1(\pi^*_1)$.

\end{proposition}
\begin{proof}
Observe  that  $\varphi_1(\Pi_1)$  only  depends on $\Pi_1$  through $\mu_1(\pi_1)$ and $\sigma^2_1(\pi_1)$, aside from   $\Phi_{h_1}(\pi_1)$. Thus,  we have \begin{align*} 
	\max\limits_{\pi_1\in \mathcal{M}(\mathbb R)}\varphi_1(\pi_1)=\max\limits_{m_1\in \mathbb R,s_1>0}\max\limits_{\substack{\pi_1\in \mathcal{M}(R)\\ \mu_1(\pi_1)=m,\sigma_1(\pi_1)^2=s^2}}\varphi_1(\pi_1),
\end{align*} and the inner maximization problem is equivalent to
\begin{align}\label{mcr}
	\max\limits_{\pi_1\in \mathcal{M}(R)}\Phi_{h_1}(\pi_1)\ \ \text{subject to }\mu_1(\pi_1)=m,\ \sigma_1(\pi_1)^2=s^2.
\end{align}
By Lemma \ref{lem:3.2}, the maximizer $\pi_1^*$ of \eqref{mcr} whose quantile function is $Q_{\pi_1^*}(p)$ satisfies
$$Q_{\pi_1^*}(p)=\mu_1(t)+\sigma_1(t)\dfrac{h_1'(1-p)}{\Vert h_1'\Vert_2}$$
and
	$\Phi_{h_1}(\pi_1^*)=\sigma_1(t)\Vert h_1'\Vert_2.$  This completes the proof. \qed
\end{proof}
Let $\mu_1^*(t)$ and $\sigma_1^*(t)$ represent  the mean and standard deviation of equilibrium response of Agent $1$.  Building upon Proposition \ref{prop:3.3}, we obtain
\begin{align*}
	(\mu_1^*(t),\sigma_1^*(t))=&\argmax\limits_{m\in \mathbb R,s>0}\left\{b(t,y)\theta(t,y)(m-k_1\mu_2(t))\right.\\
	&-\dfrac{\gamma_1}{2}b^2(t,y)[(m^2+s^2)+k_1^2(\mu_2^2(t)+\sigma_2^2(t))-2k_1m\mu_2(t)]\\
	&\left.-\gamma_1\rho v(t,y)b(t,y)(m-k_1\mu_2(t))\dfrac{\partial d_1(t,y)}{\partial y}+\lambda_1(t)s\Vert h_1'\Vert_2\right\}.
\end{align*}
By the first-order condition, we deduce that
\begin{align*}
	b(t,y)\theta(t,y)-\gamma_1b^2(t,y)\mu_1^*(t)+k_1\gamma_1\mu_2(t)b^2(t,y)-\gamma_1\rho v(t,y)b(t,y)\dfrac{\partial d_1(t,y)}{\partial y}=0,
\end{align*}
and
\begin{align*}
	-\gamma_1b^2(t,y)\sigma_1^*(t)+\lambda_1(t)\Vert h_1'\Vert_2=0.
\end{align*}
So the mean and standard deviation of equilibrium response of Agent $1$ are
\begin{align}
		\mu_1^*(t)&=\dfrac{\theta(t,y)}{\gamma_1b(t,y)}+k_1\mu_2(t)-\dfrac{\rho v(t,y)}{b(t,y)}\dfrac{\partial d_1(t,y)}{\partial y},\label{mu}\end{align}
and 
\begin{align}
		\sigma_1^*(t)&=\dfrac{\lambda_1(t)\Vert h_1'\Vert_2}{\gamma_1b^2(t,y)}.\label{sigma}
\end{align}
By substituting \eqref{mu} and \eqref{sigma} back  into \eqref{hjbg} and\eqref{simplifyhjb}, we then get that $d_1(t,y)$ and $D_1(t,y)$ respectively satisfy 
\begin{align*}
	\dfrac{\partial d_1(t,y)}{\partial t}+(m(t,y)-\rho v(t,y)\theta(t,y))\dfrac{\partial d_1(t,y)}{\partial y}+\dfrac{1}{2}v^2(t,y)\dfrac{\partial^2 d_1(t,y)}{\partial y^2}+\dfrac{\theta^2(t,y)}{\gamma_1}=0,
\end{align*}
and
\begin{align*}
	&\dfrac{\partial D_1(t,y)}{\partial t}+m(t,y)\dfrac{\partial D_1(t,y)}{\partial y}+\dfrac{1}{2}v^2(t,y)\dfrac{\partial^2(D_1-\gamma_1d_1)}{\partial y^2}(t,y)\\
	&-\dfrac{\gamma_1}{2}\rho^2v^2(t,y)\left(\dfrac{\partial d_1(t,y)}{\partial t}-1\right)^2+\dfrac{\gamma_1}{2}\rho^2v^2(t,y)-\dfrac{\theta^2(t,y)}{2\gamma_1}\\
	&-\dfrac{\gamma_1k_1}{2}b^2(t,y)\sigma_2^2(t)-\rho v(t,y)\theta(t,y)+\dfrac{\lambda_1^2(t)\Vert h_1'\Vert_2^2}{2\gamma_1b^2(t,y)}=0.
\end{align*}
Repeating the procedure outlined above can yield  results for Agent $2$. So if $(\Pi_1^*,\Pi_2^*)$ constitutes a Nash equilibrium, we have
\begin{align}\label{twomean}
	\left\{
	\begin{aligned}
		\mu_1^*(t)-k_1\mu_2^*(t)&=\dfrac{\theta(t,y)}{\gamma_1b(t,y)}-\dfrac{\rho v(t,y)}{b(t,y)}\dfrac{\partial d_1(t,y)}{\partial y},\\
		\mu_2^*(t)-k_2\mu_1^*(t)&=\dfrac{\theta(t,y)}{\gamma_2b(t,y)}-\dfrac{\rho v(t,y)}{b(t,y)}\dfrac{\partial d_2(t,y)}{\partial y}.
	\end{aligned}
	\right.
\end{align}
The solution to the system of equations can be directly calculated as
\begin{align}\label{equimean}
	\left\{
	\begin{aligned}
		\mu_1^*(t)&=\dfrac{1}{1-k_1k_2}\left[\dfrac{\theta(t,y)}{b(t,y)}\left(\dfrac{1}{\gamma_1}+\dfrac{k_1}{\gamma_2}\right)-\dfrac{\rho v(t,y)}{b(t,y)}\left(\dfrac{\partial d_1(t,y)}{\partial y}+\dfrac{k_1\partial d_2(t,y)}{\partial y}\right)\right],\\
		\mu_2^*(t)&=\dfrac{1}{1-k_1k_2}\left[\dfrac{\theta(t,y)}{b(t,y)}\left(\dfrac{1}{\gamma_2}+\dfrac{k_2}{\gamma_1}\right)-\dfrac{\rho v(t,y)}{b(t,y)}\left(\dfrac{\partial d_2(t,y)}{\partial y}+\dfrac{k_2\partial d_1(t,y)}{\partial y}\right)\right].
	\end{aligned}
	\right.
\end{align}
Summarizing the above, we get following theorem.
\begin{theorem}\label{Nashequilibrium}
	For $i\in\{1,2\}$ and $i\neq j$, let $d_i(t,y)$ and $D_i(t,y)$ be the solutions of
	\begin{align}\label{di}
		\dfrac{\partial d_i(t,y)}{\partial t}+(m(t,y)-\rho v(t,y)\theta(t,y))\dfrac{\partial d_i(t,y)}{\partial y}+\dfrac{1}{2}v^2(t,y)\dfrac{\partial^2 d_i}{\partial y^2}+\dfrac{\theta^2(t,y)}{\gamma_i}=0,
	\end{align}
	and
	\begin{align}\label{Di}
		\begin{split}
			&\dfrac{\partial D_i(t,y)}{\partial t}+m(t,y)\dfrac{\partial D_i(t,y)}{\partial y}+\dfrac{1}{2}v^2(t,y)\dfrac{\partial^2(D_i-\gamma_id_i)}{\partial y^2}(t,y)\\
	        &-\dfrac{\gamma_i}{2}\rho^2v^2(t,y)\left(\dfrac{\partial d_i(t,y)}{\partial t}-1\right)^2+\dfrac{\gamma_i}{2}\rho^2v^2(t,y)-\dfrac{\theta^2(t,y)}{2\gamma_i}\\
	        &-\dfrac{\gamma_ik_i}{2}b^2(t,y)\sigma_j^2(t)-\rho v(t,y)\theta(t,y)+\dfrac{\lambda_i^2(t)\Vert h_i'\Vert_2^2}{2\gamma_ib^2(t,y)}=0,
		\end{split}
	\end{align}
	with terminal condition $D_i(T,y)=d_i(T,y)=0$. Then $(\Pi_1^*,\Pi_2^*)$ with quantile functions
	\begin{align}\label{quantile}
	\begin{split}
		Q_{\Pi_i^*(t)}(p)=&\dfrac{1}{1-k_1k_2}\left[\dfrac{\theta(t,y)}{b(t,y)}\left(\dfrac{1}{\gamma_i}+\dfrac{k_i}{\gamma_j}\right)-\dfrac{\rho v(t,y)}{b(t,y)}\left(\dfrac{\partial d_i(t,y)}{\partial y}+\dfrac{k_i\partial d_j(t,y)}{\partial y}\right)\right]\\
		&+\dfrac{\lambda_i(t)}{\gamma_ib^2(t,y)}h_i'(1-p)
	\end{split}
	\end{align}
	is a Nash equilibrium with $p\in (0,1)$, and $V_i(t,x,y)=x+D_i(t,y)$ is the equilibrium value function of Agent $i$.
\end{theorem}

From \eqref{quantile}, it is evident   that the equilibrium distribution of  Agent $i$ is uniquely determined by his own  Choquet regularizer, $h'_i$,  and remains independent of his opponent's regularizer, $h_j$. Furthermore, \eqref{sigma} and \eqref{equimean} show that while  the mean of Agent $i$'s distribution  depends on  both his own parameters and those of his opponent,  the variance is solely determined by his own parameters, specifically $\lambda_i$, $h_i$, and $\gamma_i$. These  insights   align with intuitive expectations in the context of  RL. Although an opponent's risk tolerance, sensitivity, or strategic decisions can influence the expected outcomes of the decision-making process, the degree of exploration, as reflected by variance, is solely a function of the agent's intrinsic characteristics.  
Additionally,   \eqref{sigma} shows that  larger values of $\lambda_i$ or $h_i$ indicate a stronger emphasis on exploration,  leading to more dispersed exploration around the current position of Agent $i$. In contrast, an increase in the risk aversion parameter  $\gamma_i$
  reflects a more cautious approach, leading to reduced variance in the exploratory strategy.

\subsection{Solution to Gauss mean return model}
In this subsection, we examine the Gaussian mean return model as a special case of the state process 
$Y(t)$ shown in Example \ref{exa:1}, that is,
\begin{align*}
	a(t,y)=r+\sigma y,\ b(t,y)=\sigma,\ m(t,y)=\iota(Y-y),\ v(t,y)=v,
\end{align*}
where $r$, $\sigma$, $\iota$, $v$ and $Y$ are positive constants. Thus, by \eqref{eq:theta}, we have $\theta(t,y)=y$. 
We formulate the following proposition as a direct consequence of Theorem \ref{Nashequilibrium}.

\begin{proposition}\label{prop:3.5}
For the Gauss mean return model, $p\in (0,1)$, $i,j\in\{1,2\}$ and $i\neq j$, profile $(\Pi_1^*,\Pi_2^*)$ with quantile functions
	\begin{align}\label{quantilegauss}
		\begin{split}
			Q_{\Pi_i^*(t)}(p)=&\dfrac{1}{1-k_1k_2}\left[\dfrac{y}{\sigma}\left(\dfrac{1}{\gamma_i}+\dfrac{k_i}{\gamma_j}\right)-\dfrac{\rho v}{\sigma}\left((a_2^i(t)+k_ia_2^j(t))y+(a_1^i(t)+k_ia_1^j(t))\right)\right]\\
			&+\dfrac{\lambda_i(t)}{\gamma_i\sigma^2}h_i'(1-p),
		\end{split}
	\end{align}
	is a Nash equilibrium.
	Moreover, the corresponding equilibrium value function $V_i$ and $g_i$ have the following form
	\begin{equation}\label{vgauss}
			V_i(t,\hat x_i,y)=\hat x_i+\dfrac{1}{2}b_2^i(t)y^2+b_1^i(t)y+b_0^i(t),
	\end{equation}
	and 
	 \begin{equation} \label{ggauss} g_i(t,\hat x_i,y)=\hat x_i+\dfrac{1}{2}a_2^i(t)y^2+a_1^i(t)y+a_0^i(t), 
	\end{equation} where $a^i_n(t)$, $b^i_n(t)$,~$n=0,1,2$, are continuously differentiable functions    defined as 
	\begin{align}\label{ai}
	\left\{
	\begin{aligned}
	a_0^i(t)&=\dfrac{\iota^2 Y^2}{\gamma_i(\iota +\rho v)^2}\left(T-t-\frac{1}{2(\iota+\rho v)}e^{2(\iota+\rho v)T})+\frac{2}{(\iota+\rho v)}e^{(\iota+\rho v)T}\right)\\&~~~~+\dfrac{v^2}{2\gamma_i(\iota +\rho v)}\left(T-t+\frac{1}{2(\iota+\rho v)}e^{2(\iota+\rho v)T}\right), \\
	a_1^i(t)&=\dfrac{\iota Y}{\gamma_i(\iota +\rho v)^2}[1-e^{-(\iota+\rho v)(T-t)}]^2,\\
		a_2^i(t)&=\dfrac{1}{\gamma_i(\iota +\rho v)}[1-e^{-2(\iota+\rho v)(T-t)}],
	\end{aligned}
	\right.	
	\end{align} and 
	\begin{align}\label{dai}
		\left\{
		\begin{aligned}
			{b_2^i}'(t)=&2\iota b_2^i(t)+\dfrac{\gamma_i\rho^2v^2}{2}a_2^i(t)+\dfrac{1}{2\gamma_i},\\
			{b_1^i}'(t)=&\iota b_1^i(t)-\iota Yb_2^i(t)+\gamma_i\rho^2v^2 a_2^i(t)(a_1^i(t)-1)+\rho v,\\
			{b_0^i}'(t)=&-\iota Yb_1^i(t)-\dfrac{v^2}{2}(b_2^i(t)-\gamma_i a_2^i(t))+\dfrac{\gamma_i\rho^2v^2}{2}(a_1^i(t)^2-2a_1^i(t))\\
			&+\dfrac{\gamma_ik_i\sigma^2}{2}\sigma_j(t)^2-\dfrac{\lambda_i^2(t)\Vert h_i'\Vert_2^2}{2\gamma_i\sigma^2},
		\end{aligned}
		\right.
	\end{align} \end{proposition} with $b_0^i(T)=b_1^i(T)= b_2^i(T)=0$. 
\begin{proof}
	For the Gauss mean return model, \eqref{di} can be simplified as
	\begin{align}\label{digauss}
		\dfrac{\partial d_i(t,y)}{\partial t}+[\iota(Y-y)-\rho v y]\dfrac{\partial d_i(t,y)}{\partial y}+\dfrac{v^2}{2}\dfrac{\partial^2d_i(t,y)}{\partial y^2}+\dfrac{y^2}{\gamma_i}=0.
	\end{align}
	By letting $d_i(t,y)=\dfrac{1}{2}a_2^i(t)y^2+a_1^i(t)y+a_0^i(t)$ and substituting it into \eqref{digauss}, we obtain 	
	$$
		\left\{
		\begin{aligned}
			{a_2^i}'(t)&=2a_2^i(t)(\iota +\rho v)-\dfrac{2}{\gamma_i},& a_2^i(T)=0,\\
			{a_1^i}'(t)&=a_1^i(t)(\iota +\rho v)-a_2^i(t)\iota Y,& a_1^i(T)=0,\\
			{a_0^i}'(t)&=-a_1^i(t)\iota Y-\dfrac{v^2}{2}a_2^i(t),& a_0^i(T)=0.
		\end{aligned}
		\right.$$ 
	It can be shown that \eqref{ai} is the solution to \eqref{dai}. By substituting $d_i$ into \eqref{quantile},  we derive \eqref{quantilegauss}. Consequently,  $(\Pi_1,\Pi_2)$ is indeed a Nash equilibrium. Similarly, by simplying \eqref{Di},  we can get $D_i(t,y)=\dfrac{1}{2}b_2^i(t)y^2+b_1^i(t)y+b_0^i(t)$ with $b^i_n(t)$, $n=0,1,2$, given by \eqref{dai}. \qed
\end{proof}

The results for the complete market are straightforward when setting  $y=(a-r)/b,$ $\sigma=b$, $\iota=0$ and $v=0$, as stated below. 
\begin{corollary} In the Black-Scholes model, for $p\in (0,1)$, $i,j\in\{1,2\}$ and $i\neq j$, profile $(\Pi_1^*,\Pi_2^*)$ with quantile functions
	\begin{align}\label{eq:BS}
		\begin{split}
			Q_{\Pi_i^*(t)}(p)=&\dfrac{1}{1-k_1k_2}\left[\dfrac{a-r}{b^2}\left(\dfrac{1}{\gamma_i}+\dfrac{k_i}{\gamma_j}\right)\right]
			+\dfrac{\lambda_i(t)}{\gamma_i\sigma^2}h_i'(1-p)
		\end{split}
	\end{align}
	is a Nash equilibrium.
\end{corollary}
In the Black-Scholes model, the  influence of $k_1$, $k_2$, $\gamma_1$, and $\gamma_2$  on the equilibrium strategies is clearly reflected in \eqref{eq:BS},  aligning with the properties of   Gauss mean return model discussed  in Section \ref{num}.

\section{Policy iteration}\label{sec:4}
In this section, we employ the policy iteration method to find equilibrium strategies in two steps.   For $i, j \in \{1, 2\}$ and $i \neq j$, we first  fix $\Pi_j$ and estimate the associated value function $V_i^{\Pi_i}$ given a policy $\Pi_i$. Then we  update the previous policy $\Pi_i$ to a new one $\tilde{\Pi}_i$ based on the obtained value function $V_i^{\Pi_i}$.  Despite the learning process not leading to a monotone iteration algorithm due to the ``optimality"  is in  the sense of  equilibrium,  we demonstrate that the iterative process converges uniformally  to the desired equilibrium policy.

 Assuming $\Pi_j$ is fixed,  and letting $\Pi_i$ be an admissible strategy for $i, j \in \{1,2\}$ with $i \neq j$, we denote the value function under $\Pi_i$  as $V_i^{\Pi_i}(t,\hat x_i,y)=J_i(t,\hat x_i,y;\Pi_i,\Pi_j)$. Similarly, we define $g_i^{\Pi_i}(t,\hat x_i,y)=\mathbb E_t[\hat X_i^{\Pi_i,\Pi_j}(T)]$. 
   According to \cite{BKM17},   the functions $V_i^{\Pi_i}$ and  $g_i^{\Pi_i}$ satisfy the following equations
\begin{align}\label{fmv}
	&\mathcal L^{\pi_i,\Pi_j}V^{\Pi_i}_i-\dfrac{\gamma_i}{2}\mathcal L^{\pi_i,\Pi_j}{g^{\Pi_i}_i}^2+\gamma_ig^{\Pi_i}_i\mathcal L^{\pi_i,\Pi_j}g^{\Pi_i}_i+\lambda_i(t)\Phi_{h_i}(\pi_i)=0,
\end{align}
and
\begin{align}\label{fmg}
	\mathcal L^{\pi_i,\Pi_j}g^{\Pi_i}_i(t,\hat x_i,y)=0, 
\end{align}
with
\begin{align}
	V^{\Pi_i}_i(T,\hat x_i,y)=\hat x_i,~~\ g^{\Pi_i}_i(T,\hat x_i,y)=\hat x_i.
\end{align}

\begin{theorem}\label{responceconvergence}
	For  $p\in (0,1)$, $i,j\in\{1,2\}$ and $i\neq j$, with $\Pi_j$  fixed,  let $\Pi_i^0$ be the initial policy of Agent $i$ with quantile function given by
	\begin{align}\label{pi0}
		Q_{\Pi_i^0(t)}(p)=\dfrac{y}{\gamma_i\sigma}+k_i\mu_j(t)-\dfrac{\rho v}{\sigma}(a_2^{i0}(t)y+a_1^{i0}(t))+\theta^0 h_i'(1-p).\end{align} 
	Choose one policy $$\pi_i\in \argmax\limits_{\pi_i\in\mathcal M}\{\mathcal L^{\pi_i,\Pi_j}V^{\Pi_i^n}_i-\dfrac{\gamma_i}{2}\mathcal L^{\pi_i,\Pi_j}{g^{\Pi_i^n}_i}^2+\gamma_ig^{\Pi_i^n}_i\mathcal L^{\pi_i,\Pi_j}g^{\Pi_i^n}_i+\lambda_i(t)\Phi_{h_i}(\pi_i)\},$$
	and denote this policy as $\Pi_i^{n+1}$, $n=0,1,2,...$ Then the following statements holds.
	\begin{enumerate}[(i)]
		\item The sequence of updated policies $\Pi_i^{n}$ for $n\geqslant 1$ has the quantile function
	    \begin{align}
	    	Q_{\Pi_i^n(t)}(p)=\dfrac{y}{\gamma_i\sigma}+k_i\mu_j(t)-\dfrac{\rho v}{\sigma}(a_2^{in}(t)y+a_1^{in}(t))+\dfrac{\lambda_i(t)}{\gamma_i\sigma^2} h_i'(1-p),\label{pinquantile}
	    \end{align}
	    where $a_1^{in}$ and $a_2^{in}$ satisfy
	    \begin{align}\label{pina}
	    	\left\{
		\begin{aligned}
			{a_2^{in}}'(t)&=2\iota a_2^{in}(t)+2\rho va_2^{in-1}(t) -\dfrac{2}{\gamma_i},&a_2^{in}(T)=0,\\
			{a_1^{in}}'(t)&=\iota a_1^{in}(t) +\rho va_1^{in-1}(t)-a_2^{in}(t)\iota Y,& a_1^{in}(T)=0.
		\end{aligned}
		\right.
	    \end{align}
	    \item As $n\to \infty$, $a_1^{in}(t)$ and $a_2^{in}$ uniformly converge to $a_1^i$ and $a_2^i$ in \eqref{ai}, respectively.
	\end{enumerate}
\end{theorem}
\begin{proof}
	(i) Note that $\Pi_i^0$ satisfies
	$$\begin{aligned}\label{pi0v}
	&\mathcal L^{\pi_i,\Pi_j}V^{\Pi^0_i}_i-\dfrac{\gamma_i}{2}\mathcal L^{\pi_i,\Pi_j}({g^{\Pi^0_i}_i})^2+\gamma_ig^{\Pi^0_i}_i\mathcal L^{\pi_i,\Pi_j}g^{\Pi^0_i}_i+\lambda_i(t)\Phi_{h_i}(\pi_i)=0,
    \end{aligned}$$
    and
    \begin{align}\label{pi0g}
	    \mathcal L^{\pi_i,\Pi_j}g^{\Pi^0_i}_i(t,\hat x_i,y)=0.
    \end{align}
    Consider $V^{\Pi^0_i}_i(t,\hat x_i,y)=\hat x_i+D^{\Pi^0_i}_i(t,y)$ and $g^{\Pi^0_i}_i(t,\hat x_i,y)=\hat x_i+d^{\Pi^0_i}_i(t,y)$.  Substituting $g^{\Pi^0_i}_i$ into \eqref{pi0g}, we get
    \begin{align}\label{pi0d}
    	\dfrac{\partial d_i^{\Pi^0_i}(t,y)}{\partial t}+\iota(Y-y)\dfrac{\partial d_i^{\Pi^0_i}(t,y)}{\partial y}+\dfrac{v^2}{2}\dfrac{\partial^2d_i^{\Pi^0_i}(t,y)}{\partial y^2}+\dfrac{y^2}{\gamma_i}-\rho v(a_2^{i0}(t)y^2+a_1^{i0}(t)y)=0.
    \end{align}
    Assuming  $d_i^{\Pi^0_i}(t,y)=\dfrac{1}{2}a_2^{i1}(t)y^2+a_1^{i1}(t)y+a_0^{i1}(t)$ and substituting it into \eqref{pi0d}, we get
    \begin{align}\label{pi0a}
		\left\{
		\begin{aligned}
			{a_2^{i1}}'(t)&=2\iota a_2^{i1}(t)+2\rho va_2^{i0}(t) -\dfrac{2}{\gamma_i},&a_2^{i1}(T)=0,\\
			{a_1^{i1}}'(t)&=\iota a_1^{i1}(t) +\rho va_1^{i0}(t)-a_2^{i1}(t)\iota Y,& a_1^{i1}(T)=0,\\
			{a_0^{i1}}'(t)&=-a_1^{i1}(t)\iota Y-\dfrac{v^2}{2}a_2^{i1}(t),&a_0^{i1}(T)=0.
		\end{aligned}
		\right.
	\end{align}
	By policy iteration, we know that
	\begin{align*}
		\Pi_i^1(t)\in \argmax\limits_{\pi_i\in\mathcal M}\left\{\mathcal L^{\pi_i,\Pi_j}V^{\Pi_i^0}_i-\dfrac{\gamma_i}{2}\mathcal L^{\pi_i,\Pi_j}{g^{\Pi_i^0}_i}^2+\gamma_ig^{\Pi_i^0}_i\mathcal L^{\pi_i,\Pi_j}g^{\Pi_i^0}_i+\lambda_i(t)\Phi_{h_i}(\pi_i)\right\}.
	\end{align*}
	By the first order conditions,  we have 
	\begin{align}\label{sigmagauss}
		\mu_i^1(t)
		=\dfrac{y}{\gamma_i\sigma}+k_i\mu_j(t)-\dfrac{\rho v}{\sigma}(a_2^{i1}(t)y+a_1^{i1}(t)),~~\text{and}~~
		\sigma_i^1(t)&=\dfrac{\lambda_i(t)\Vert h_i'\Vert_2}{\gamma_i\sigma^2}. \end{align} Repeating the above procedure, we then get \eqref{pinquantile} and \eqref{pina}.
	
	(ii) Denote $M=\sup\limits_{t\in[0,T]}|a_2^i(t)-a_2^{i0}(t)|$, $m=\sup\limits_{t\in[0,T]}|a_1^i(t)-a_1^{i0}(t)|$, $\Delta_{k+1}(t)=a_2^i(t)-a_2^{i(k+1)}(t)$ and $\delta_{k+1}(t)=a_1^i(t)-a_1^{i(k+1)}(t)$. We claim that
	\begin{align}\label{delta}
		|\Delta_n(t)|\leqslant \dfrac{[2\rho v(T-t)]^n}{n!}M.
	\end{align}
	The case for  $n=0$ is trivial. By induction, we assume that the inequality holds for $n=K$. Then it follows from \eqref{dai} and \eqref{pina} that $\Delta_{k+1}(t)$ satisfies
	\begin{align*}
		\Delta_{k+1}'(t)=2\iota \Delta_{k+1}(t)+2\rho v\Delta_k(t),\quad \Delta_{k+1}(T)=0.
	\end{align*}
	Solving this differential equation, we obtain $\Delta_{k+1}(t)=-\int_t^T2\rho ve^{2\iota(t-s)}\Delta_k(s)\d s$. Consequently,
	\begin{align*}
		|\Delta_{k+1}(t)|\leqslant\int_t^T2\rho v|\Delta_{k}(s)|\d s\leqslant\int_t^T2\rho v\dfrac{[2\rho v(T-s)]^k}{k!}M\d s=\dfrac{[2\rho v(T-t)]^{k+1}}{(k+1)!}M.
	\end{align*}
	Thus, \eqref{delta} holds. Similarly, we can prove by induction that
	\begin{align*}
		|\delta_n(t)|\leqslant\dfrac{[\rho v(T-t)]^n}{n!}m+\dfrac{\iota Y}{\rho v} \dfrac{[2\rho v(T-t)]^{n+1}}{(n+1)!}M.
	\end{align*}
	Thus, $a_1^{in}(t)$ and $a_2^{in}$ uniformly converge to $a_1^i$ and $a_2^i$ as $n\to \infty$, respectively.\qed
\end{proof}
Theorem \ref{responceconvergence}  shows that  the iteration does not change the form of the policy (see \eqref{pinquantile}),  and thus it suffices to parameterize the iterative policy through two deterministic functions $(a_1^{(in)}(t), a_2^{(in)}(t))$.   In particular, when the initial policy  chosen as the form of the equilibrium policy  in Proposition \ref{prop:3.5}, our algorithm is guaranteed to uniformly converge to the equilibrium policy. 

 The next result guarantees  the convergence of policies as the two agents iterate simultaneously.
\begin{theorem}\label{fixpoint}
	For an initial profile $(\Pi^0_1,\Pi^0_2)$, assume that two agents iterate simultaneously by \eqref{pinquantile} and the updated sequence is defined by $(\Pi_1^n,\Pi_2^n)$, $n=0,1,2,...$ Then for  $p\in (0,1)$, $i,j\in\{1,2\}$ and $i\neq j$, $Q_{\Pi_i^n(t)}(p)$ converges uniformly to $Q_{\Pi_i^*(t)}(p)$ of \eqref{quantilegauss} as $n\to\infty$.
\end{theorem}
\begin{proof}
	We just need to prove that the mean and variance of $\Pi_i$ converge to the mean and variance of $\Pi_i^*$. The convergence of variance is obvious according to the proof of Theorem \ref{responceconvergence} (i) and \eqref{sigmagauss}. Let $\mu_i^n(t)$ be the mean of $\Pi_i^n(t)$. Based on \eqref{mu}, we have
	\begin{align*}
		\begin{bmatrix}
          \mu_1^{n+1}(t)\\
          \mu_2^{n+1}(t)
        \end{bmatrix}=
        \begin{bmatrix}
        	0 & k_1 \\
        	k_2 & 0 
        \end{bmatrix}
        \begin{bmatrix}
        	\mu_1^{n}(t)\\
            \mu_2^{n}(t)
        \end{bmatrix}+
        \begin{bmatrix}
        	\dfrac{y}{\gamma_1\sigma}-\dfrac{\rho v}{\sigma}(a_2^1(t)y+a_1^1(t))\\[4mm]
        	\dfrac{y}{\gamma_2\sigma}-\dfrac{\rho v}{\sigma}(a_2^2(t)y+a_1^2(t))
        \end{bmatrix}.
	\end{align*}
	Consider the normed  space $\mathbb R^2$  with  $\Vert\cdot\Vert$ defined as  $\Vert \vec{x}\Vert=\max\{x_1,x_2\}$ for  $\vec{x}=[x_1,x_2]'\in \mathbb R^2$. It is well known that $(\mathbb R^2,\Vert\cdot\Vert)$ is a Banach space. Define 
	\begin{align*}
		f(\vec{x})=\begin{bmatrix}
        	0 & k_1 \\
        	k_2 & 0 
        \end{bmatrix}
        \vec{x}+
        \begin{bmatrix}
        	\dfrac{y}{\gamma_1\sigma}-\dfrac{\rho v}{\sigma}(a_2^1(t)y+a_1^1(t))\\[4mm]
        	\dfrac{y}{\gamma_2\sigma}-\dfrac{\rho v}{\sigma}(a_2^2(t)y+a_1^2(t))
        \end{bmatrix}.
	\end{align*}
	We have $\Vert f(\vec{x})-f(\vec{y})\Vert\leqslant \max\{k_1,k_2\}\Vert \vec{x}-\vec{y} \Vert$, establishing $f$ as a contraction mapping with a unique fixed point. By \eqref{twomean},   the fixed point is precisely $[\mu_1^*(t),\mu_2^*(t)]'$. Thus,  the mean of $\Pi_i^n(t)$ also converges. Moreover, let $\omega$ be an uniformly upper bound of $|\mu_i^0(t)-\mu_i^*(t)|$. Then we have
	\begin{align*}
		\left\Vert\begin{bmatrix}
          \mu_1^{n}(t)\\
          \mu_2^{n}(t)
        \end{bmatrix}-
        \begin{bmatrix}
          \mu_1^{*}(t)\\
          \mu_2^{*}(t)
        \end{bmatrix}\right\Vert=&
        \left\Vert\begin{bmatrix}
        	0 & k_1 \\
        	k_2 & 0 
        \end{bmatrix}
        \begin{bmatrix}
          \mu_1^{n-1}(t)-\mu_1^*(t)\\
          \mu_2^{n-1}(t)-\mu_2^*(t)
        \end{bmatrix}\right\Vert\\
        =&\left\Vert\begin{bmatrix}
        	0 & k_1 \\
        	k_2 & 0 
        \end{bmatrix}^n
        \begin{bmatrix}
          \mu_1^{0}(t)-\mu_1^*(t)\\
          \mu_2^{0}(t)-\mu_2^*(t)
        \end{bmatrix}\right\Vert\\
        \leqslant & \omega(\max\{k_1,k_2\})^n.
	\end{align*}
	Thus, for $p\in (0,1)$,  $i,j\in\{1,2\}$ and $i\neq j$, $Q_{\Pi_i^n(t)}(p)$ converges uniformly to $Q_{\Pi_i^*(t)}(p)$ of \eqref{quantilegauss} as $n\to\infty$.
	\qed
\end{proof}
\section{RL algorithm design}\label{sec:5}
In this section, we devise an algorithm to learn the Nash equilibrium. As mentioned in the introduction, game scenarios involve multiple agents in the environment, requiring the utilization of multi-agent reinforcement learning algorithms, which inherently introduce greater complexity. Specifically, in single-agent reinforcement learning, a basic assumption is the stability of the environment, wherein the transition probability and reward function remain constant. However, when other intelligent agents are introduced into the environment, this assumption no longer holds true. In a multi-agent context, any change in one agent's strategy can significantly impact other agents, leading to dynamic evolution of the environment with their strategies. Moreover, as the number of agents increases, the complexity of training also escalates. 
Fortunately, in our model, one agent will only affect the mean of the other agent's strategy, and the difference $\mu_i(t)-k_i\mu_j(t)$  for $i,j\in\{1,2\}$ and $i\neq j$ is fixed no matter how the strategies of the two agents change based on \eqref{equimean}. So, once we know the difference, we can directly use the learning procedure in Theorem \ref{fixpoint}. Thus, we can simplify this game problem into two independent optimization problems, and then apply the method proposed by \cite{DDJ23} to learn the difference.  Next we briefly introduce this method.

Assume that the risk-free interest rate $r$, the risk-aversion coefficient $\gamma_1, \gamma_2$, the sensitivity coefficient $k_1, k_2$ and the exploration weight $\lambda_1, \lambda_2$ are known. The agents do not have any information about $S(t)$ and $Y(t)$ but can observe $(S(t),Y(t))$ at time $t$. In the continuous time setting, we should discretize $[0,T]$ into $N$ intervals with equal length $\Delta t=t_{k+1}-t_k, k=0,1,...,N-1$ first. From \eqref{dynamics} and \eqref{dynamicx}, we get
\begin{align*}
	\d X_i^{u_i}(t)=u_i(t)\dfrac{\d e^{-rt}S(t)}{e^{-rt}S(t)}.
\end{align*}
Therefore, when the Agent $i$ follows strategy $\Pi_i(t_k)$ at time $t_k$ and samples action $u_i(t_k)$ from $\Pi_i(t_k)$, the discounted wealth at $t_{k+1}$ is
\begin{align*}
	X_i^{\Pi_i}(t_{k+1})\approx X_i^{\Pi_i}(t_k)+u_i(t_k)\dfrac{e^{-rt_{k+1}}S(t_{k+1})-e^{-rt_k}S(t_k)}{e^{-rt_k}S(t_k)}.
\end{align*}

The algorithm to learn the difference $\mu_i(t)-k_i\mu_j(t)$ is based on the basic idea of Actor-Critic algorithm in \cite{KT99}. Assume that the strategy $\Pi_j$ of Agent $j$ is fixed. By Theorem \ref{responceconvergence},  we can parameterize \eqref{pi0} with $\Phi=(\phi_0,\phi_1,\phi_2,\phi_3)$ to
\begin{align}\label{paraq}
Q_{\Pi_i^{\Phi}(t)}(p)=k_i\mu_j(t)+\phi_0y-\phi_1\dfrac{1-e^{-2\phi_2(T-t)}}{\phi_2}y-\phi_3\dfrac{(1-e^{-\phi_2(T-t)})^2}{\phi_2^2}+\lambda_i\phi_0^2\gamma_ih_i'(1-p),	
\end{align}
as the ``Actor". At the same time, due to Proposition \ref{prop:3.5}, we can parameterize \eqref{vgauss} and \eqref{ggauss} respectively to be the ``Critic"  as follows
\begin{align}\label{paravg}
	\begin{split}
		V_i^{\Theta}(t,\hat x_i,y)&=\hat x_i+p(\theta_i^{V,2},T-t)y^2+p(\theta_i^{V,1},T-t)y+p(\theta_i^{V,0},T-t),\\
		g_i^{\Theta}(t,\hat x_i,y)&=\hat x_i+p(\theta_i^{g,2},T-t)y^2+p(\theta_i^{g,1},T-t)y+p(\theta_i^{g,0},T-t),
	\end{split}
\end{align}
where $p(\theta,t)$ is an appropriate function with parameter $\theta\in \mathbb R^d$  for approximating a continuous function of $t$ and $\Theta=(\theta_i^{V,0},\theta_i^{V,1},\theta_i^{V,2},\theta_i^{g,0},\theta_i^{g,1},\theta_i^{g,2})\in\mathbb R^{6d}$. Usually, one can choose the sum of the first $d$ terms of the Taylor series or Fourier series as $p(\theta,t)$, and the parameter $\theta$ is the coefficient corresponding to each term.

Building on the formal consistency observed in  Theorem  \ref{responceconvergence} (i),  and using  the parameterized expression, we can now  proceed to update both the Actor and Critic. By  applying Itô's fomula, for any $\varphi\in C^{1,2,2}$ and strategy $\Pi_i$ of Agent $i$, we have
\begin{align*}
\begin{split}
	&\mathbb E_t[\varphi(t+\Delta t,\hat X_i^{\Pi_i,\Pi_j}(t+\Delta t),Y(t+\Delta t))]-\varphi(t,\hat X_i^{\Pi_i,\Pi_j}(t),Y(t))\\
	& =\mathbb E_t\int_t^{t+\Delta t}\mathcal L^{\Pi_i,\Pi_j}\varphi(s,\hat X_i^{\Pi_i,\Pi_j}(s),Y(s))\d s,
\end{split}
\end{align*}
and thus
\begin{align}\label{infgen}
	\mathcal L^{\Pi_i,\Pi_j}\varphi(t,\hat X_i^{\Pi_i,\Pi_j}(t),Y(t))\approx \dfrac{\varphi(t+\Delta t,\hat X_i^{\Pi_i,\Pi_j}(t+\Delta t),Y(t+\Delta t))-\varphi(t,\hat X_i^{\Pi_i,\Pi_j}(t),Y(t))}{\Delta t}.
\end{align}
When Agent $i$ employs  the  strategy $\Pi_i^{\Phi}$, the functions $V_i^{\Pi^{\Phi}_i}$ and  $g_i^{\Pi^{\Phi}_i}$ satisfy the dynamic equations given in \eqref{fmv} and \eqref{fmg}, respectively. We can use \eqref{infgen} to approximate \eqref{fmv} and \eqref{fmg} by replacing $\varphi$ with $V_i^{\Pi^{\Phi}_i}$, $g_i^{\Pi^{\Phi}_i}$ and  $(g_i^{\Pi^{\Phi}_i})^2$. Specifically, we define the following temporal-difference (TD) error terms 
\begin{align}\label{c1}
	\begin{split}
		C_i^{1(k)}(\Phi,\Theta)=&\dfrac{V_i^{\Theta}(t_{k+1},\hat X_i(t_{k+1}),Y(t_{k+1}))-V_i^{\Theta}(t_{k},\hat X_i(t_{k}),Y(t_{k}))}{\Delta t}\\
		&+\gamma_ig_i^{\Theta}(t_{k},\hat X_i(t_{k}),Y(t_{k}))\dfrac{g_i^{\Theta}(t_{k+1},\hat X_i(t_{k+1}),Y(t_{k+1}))-g_i^{\Theta}(t_{k},\hat X_i(t_{k}),Y(t_{k}))}{\Delta t}\\
		&-\dfrac{\gamma_i}{2}\dfrac{g_i^{\Theta}(t_{k+1},\hat X_i(t_{k+1}),Y(t_{k+1}))^2-g_i^{\Theta}(t_{k},\hat X_i(t_{k}),Y(t_{k}))^2}{\Delta t}\\
		&+\lambda_i\Phi_{h_i}(\Pi^{\Phi}_i(t_k)),
	\end{split}
\end{align}
and
\begin{align}\label{c2}
	C_i^{2(k)}(\Phi,\Theta)=\dfrac{g_i^{\Theta}(t_{k+1},\hat X_i(t_{k+1}),Y(t_{k+1}))-g_i^{\Theta}(t_{k},\hat X_i(t_{k}),Y(t_{k}))}{\Delta t}.
\end{align}
In fact, $C_i^{1(k)}$ and $C_i^{2(k)}$ are extensions of the TD  error in \cite{D20} and \cite{SB18} to the time inconsistent problem.  The purpose of updating Critic is to select appropriate parameters $\Theta$ to make $V_i^{\Theta}$ and $g_i^{\Theta}$ sufficiently close to $V_i^{\Pi^{\Phi}}$ and $g_i^{\Pi^{\Phi}}$, respectively. This is achieved by minimizing the functions  $$\sum\limits_{k=0}^{N-1}C_i^{1(k)}(\Theta,\Phi)^2, ~\text{and} ~\sum\limits_{k=0}^{N-1}C_i^{2(k)}(\Theta,\Phi)^2.$$ Note that calculating the partial derivative $\partial C_i^{1(k)}/\partial \Theta$ and $\partial C_i^{1(k)}/\partial \Theta$ is feasible due to \eqref{paravg}. These gradients facilitate the calculation of  $\nabla \Theta$, enabling optimization through stochastic gradient descent.

As for updating Actor, we apply the   iteration  procedure outlined Theorem \ref{responceconvergence}.  We define
\begin{align*}
	L_i(\Phi;t,\hat x,y):=\mathcal L^{\Pi^{\Phi}_i(t),\Pi_j}V^{\Theta}_i-\dfrac{\gamma_i}{2}\mathcal L^{\Pi^{\Phi}_i(t),\Pi_j}{g^{\Theta}_i}^2+\gamma_ig^{\Theta}_i\mathcal L^{\Pi^{\Phi}_i(t),\Pi_j}g^{\Theta}_i+\lambda_i\Phi_{h_i}(\Pi^{\Phi}_i(t)). 
\end{align*}
Theorem \ref{responceconvergence} requires us to maximize $L_i(\Phi;t,\hat x,y)$ for all possible $t,\hat x,y$ in the domain of $\Phi$.  However, since only sampled data  $(t_k,\hat X_i(t_k),Y(t_k))$ is available, we approximate this by maximizing $\sum\limits_{k=0}^{N-1}L_i(\Phi;t_k,\hat X_i(t_k),Y(t_k))$ instead. We also use gradient descent method to realize maximization. To calculate the gradient of $\nabla_{\Phi}L_i(\Phi)$, we adopt the smoothed functional gradient method described in  \cite[Chapter 6]{BPP13}. The smoothed gradient is defined as \begin{align}\label{sf}
	D_{\kappa}L_i(\Phi)=\int G_{\kappa}(\Phi-\vec{\eta})\nabla_{\vec{\eta}}L_i(\vec{\eta})d\vec{\eta},
\end{align}
where $G_{\kappa}$ is the density of $\mathcal N(\vec{0},\kappa^2I_4)$.  Thus, $D_{\kappa}L_i(\Phi)$ is a smoothed gradient of $L_i(\Phi)$ and $\lim\limits_{\kappa\rightarrow 0}D_{\kappa}L_i(\Phi)=\nabla_{\Phi}L_i(\Phi)$. Applying integration by parts to \eqref{sf}, we get
\begin{align*}
	D_{\kappa}L_i(\Phi)&=-\int \nabla_{\vec{\eta}}G_{\kappa}(\Phi-\vec{\eta})L_i(\vec{\eta})\d\vec{\eta}\\
	&=\int \nabla_{\vec{\eta}}G_{\kappa}(\vec{\eta})L_i(\Phi-\vec{\eta})\d\vec{\eta}\\
	&=-\int \dfrac{\vec{\eta}}{\kappa^2} G_{\kappa}(\vec{\eta})L_i(\Phi-\vec{\eta})\d\vec{\eta}\\
	&=-\dfrac{1}{\kappa}\int \vec{\eta} G_{1}(\vec{\eta})L_i(\Phi-\kappa\vec{\eta})d\vec{\eta}\\
	&=\mathbb E\left[\dfrac{1}{\kappa}ZL_i(\Phi+\kappa Z)\right],
\end{align*}
where  $Z$ is a standard multivariate Gaussian stochastic vector. Therefore,  when $\kappa$ is small enough, we can obtain an approximation of the gradient
\begin{align*}
	\nabla_{\Phi}L_i(\Phi)\approx \dfrac{1}{\kappa}ZL_i(\Phi+\kappa Z).
\end{align*}
Further, we can reduce the variance by
\begin{align}\label{gradientphi}
	\nabla_{\Phi}L_i(\Phi)\approx \dfrac{1}{\kappa}Z(L_i(\Phi+\kappa Z)-L_i(\Phi)).
\end{align}
It is obvious that $L_i(\Phi;t_k,\hat X_i(t_k),Y(t_k))\approx C_i^{1(k)}(\Phi,\Theta)$ by \eqref{c1} and the definition of $L_i$. To calculate $L_i(\Phi+\kappa Z;t_k,\hat X_i(t_k),Y(t_k))$, let $\overline\Phi=\Phi+\kappa Z$ be the perturbed parameters. Sampling action $\overline u_i(t_k)$ from $\Pi^{\overline\Phi}_i(t_k)$, we obtain
\begin{align*}
	\overline X_i^{\Pi_i}(t_{k+1})\approx X_i^{\Pi_i}(t_k)+\overline u_i(t_k)\dfrac{e^{-rt_{k+1}}S(t_{k+1})-e^{-rt_k}S(t_k)}{e^{-rt_k}S(t_k)}.
\end{align*}
Define $\hat{\overline X}_i(t_k)=\overline X_i^{\Pi_i}(t_{k})-k_iX_j^{\Pi_j}(t_{k})$, we have
\begin{align*}
	\begin{split}
		L_i(\overline\Phi)\approx &C_i^{1(k)}(\overline\Phi,\Theta)\\
		=&\dfrac{V_i^{\Theta}(t_{k+1},\hat{\overline X}_i(t_{k+1}),Y(t_{k+1}))-V_i^{\Theta}(t_{k},\hat X_i(t_{k}),Y(t_{k}))}{\Delta t}\\
		&+\gamma_ig_i^{\Theta}(t_{k},\hat X_i(t_{k}),Y(t_{k}))\dfrac{g_i^{\Theta}(t_{k+1},\hat {\overline X}_i(t_{k+1}),Y(t_{k+1}))-g_i^{\Theta}(t_{k},\hat X_i(t_{k}),Y(t_{k}))}{\Delta t}\\
		&-\dfrac{\gamma_i}{2}\dfrac{g_i^{\Theta}(t_{k+1},\hat {\overline X}_i(t_{k+1}),Y(t_{k+1}))^2-g_i^{\Theta}(t_{k},\hat X_i(t_{k}),Y(t_{k}))^2}{\Delta t}\\
		&+\lambda_i\Phi_{h_i}(\Pi^{\overline\Phi}_i(t_k)).
	\end{split}
\end{align*}
 We can  then calculate the gradient $\nabla_{\Phi}L_i(\Phi)$ by \eqref{gradientphi}, and update the Actor using the Adam algorithm in \cite{KB14}.
In summary, we have Algorithm 1.
\begin{algorithm}
	\caption{}	
	\textbf{Input:} initial wealth $x_1,x_2$, risk-free interest rate $r$, exploration weight $\lambda_1,\lambda_2$, the parameters ($\sigma,\iota,Y,v,\rho$) of Market, investment horizon $T$, time step $\Delta t$, number of time grids $N$, learning rates $\alpha$, number of iterations $M$, smoothing functional parameter $\kappa$, risk-aversion coefficient $\gamma_1,\gamma_2$, sensitivity coefficient $k_1,k_2$ and a simulator of the market called $Market$.\\
	\textbf{Learning procedure:}
	Initialize $\Theta,\Phi$.
	\begin{algorithmic}
		\FOR {$m=1$ \textbf{to} $M$}
		\STATE Initialize $n=0$.
		\STATE $x_i(t_n) \leftarrow x_i,x_j(t_n) \leftarrow x_j,\overline x_i(t_n) \leftarrow x_i$.
		\STATE Compute and store $\hat x_i(t_n)=x_i(t_n)-k_ix_j(t_n)$ and $\hat{\overline x}_i(t_n)=\overline x_i(t_n)-k_ix_j(t_n)$.
		\WHILE {$n<N$}
		\STATE Sample $u_j(t_n)$ from $\Pi_j(t_n)$.
		\STATE Apply $u_j(t_n)$ to the market simulator and get the state $x_j$ of Agent $j$ at $t_{n+1}$.
		\STATE Store $x_j(t_{n+1})\leftarrow x_j$.
		\STATE Sample $u$ from uniform distribution on [0,1].
		\STATE Sample $z_n$ from $\mathcal N(\vec{0},I_4)$ and compute $\overline\Phi=\Phi+\kappa z$.
		\STATE Use $u$ to generate $u_i(t_n)$ and $\overline u_i(t_n)$ based on \eqref{paraq}.
		\STATE Apply $u_i(t_n)$ and $\overline u_i(t_n)$ to the market simulator and get the state $x_i$ and $\overline x_i$ at $t_{n+1}$. 
		\STATE Store $x_i(t_{n+1})\leftarrow x_i$ and$\overline x_i(t_{n+1})\leftarrow \overline x_i$.
		\STATE Compute $\hat x_i(t_{n+1})=x_i(t_{n+1})-k_ix_j(t_{n+1})$ and $\hat{\overline x}_i(t_{n+1})=\overline x_i(t_{n+1})-k_ix_j(t_{n+1})$.
		\STATE $n\leftarrow n+1$.
		\ENDWHILE
		\STATE Compute $V_i^{\Theta}$ and $g_i^{\Theta}$.
		\STATE Compute the TD error $C_i^{1(k)}(\Phi,\Theta)$ and $C_i^{2(k)}(\Phi,\Theta)$ for all $k$.
		\STATE Compute the gradient $\nabla \Theta$ of $\sum\limits_{k=0}^{N-1}C_i^{1(k)}(\Theta,\Phi)^2$ and $\sum\limits_{k=0}^{N-1}C_i^{2(k)}(\Theta,\Phi)^2$.
		\STATE Update $\Theta\leftarrow \Theta-\alpha\nabla\Theta$.
		\STATE Compute the TD error $C_i^{1(k)}(\Phi,\Theta)$ and $C_i^{1(k)}(\overline\Phi,\Theta)$ for all $k$.
		\STATE Compute the gradient $\nabla\Phi=\sum\limits_{k=0}^{N-1}\dfrac{z_k}{\kappa}(C_i^{1(k)}(\overline\Phi,\Theta)-C_i^{1(k)}(\Phi,\Theta))$
		\STATE Update $\Phi$ by the Adam algorithm.
		\ENDFOR
	\end{algorithmic}
\label{alg:1} \end{algorithm}

\section{Numerical results}\label{num}

Before we proceed with exploration, given the variety of Choquet regularizers available, it is possible to select different regularizers for each agent. In equation \eqref{eq:quan_rep}, $h'(x)$ represents the ``probability weight" assigned to $x$ when calculating the (nonlinear) Choquet expectation (see, e.g., \cite{GS89} and \cite{Q82}). Consequently, the choice of the distortion function $h$  can directly influence the agent's attitude toward risk. As shown by \cite{HWZ23}, Choquet regularizers can generate several widely used exploratory samplers, such as the $\epsilon$-greedy strategy, exponential, uniform and Gaussian.  Below, we assume that the agents adopt different Choquet regulizers, resulting in their optimal exploration distributions being normal and uniform, respectively.

Assume that   Agent $1$  applies the  the Choquet regulizer \begin{equation}\label{Normal}\Phi_{h_1} (\Pi_1)=\int_0^1 Q_{\Pi_1}(p) z (p) \d p,\end{equation}  where  
$z$ is the quantile function of a standard normal distribution, yielding  $h_1(p)=\int_0^p z(1-s)\d s$ with  $p\in [0,1]$. 
Further, Agent $2$ uses the Choquet regulizer   \begin{equation}\label{Gini}\Phi_{h_2}   (\Pi_2) =\frac 12 \mathbb{E}[|X_1-X_2|],\end{equation} in which  $X_1$ and $X_2$ are two iid copies from the distribution $\Pi_2$. It is 
known as the \emph{Gini mean difference} (e.g., \cite{FWZ17} with   $h_2(p)= p-p^2$ with  $p\in [0,1]$.  Let $$\mu_i^*(t)= \dfrac{1}{1-k_1k_2}\left[\dfrac{y}{\sigma}\left(\dfrac{1}{\gamma_i}+\dfrac{k_i}{\gamma_j}\right)-\dfrac{\rho v}{\sigma}\left((a_2^i(t)+k_ia_2^j(t))y+(a_1^i(t)+k_ia_1^j(t))\right)\right],$$ where  $a^i_n(t)$,~$n=1,2$,  are given by \eqref{ai}.   Based on \eqref{quantile},   the   equilibrium distribution $\Pi_1^*$ is a  normal  distribution   given as
$$
{\Pi}_1^{*} (t)=  {\mathrm N}\left(\mu_1^*(t), \dfrac{\lambda_1(t)}{\gamma_1\sigma^2}\right),
$$  and the equilibrium distribution  $\Pi_2^*$ is  a uniform  distribution   given as
$$\mathrm{U}\left[ u_2^*(t)-\dfrac{\lambda_2(t)}{3\gamma_2\sigma^2}, 
		u_2^*(t)+\dfrac{\lambda_2(t)}{3\gamma_2\sigma^2}	 \right]. $$
Next, we investigate the influence of parameters $k_i$, $\gamma_i$ and $t$  on the  equilibrium strategies of  both agents. We assume that   $\lambda_1(t)=\lambda_2(t) = \lambda_0 e^{\lambda_0(T-t)}$.   Unless otherwise specified, the parameters in \eqref{quantile} are set as in Table \ref{tab:parameters}.\footnote{For the Gauss mean return model, \cite{W02} estimates the market parameters, and \cite{DDJ23} uses these parameters for their algorithm.  In this work, we also adopt these parameters for our analysis.
}  
\begin{table}[ht]
\centering
\caption{Parameter values used in the modelh}
\begin{tabular}{|c|c|c|c|c|c|c|c|c|c|c|c|c|c|}
\hline
$\rho$ & $r$ & $\sigma$ & $\iota$ & $v$ & $Y$ & $\gamma_1$ & $\gamma_2$ & $k_1$ & $k_2$ & $\lambda_0$ & $T$ \\ 
\hline
-0.93 & 0.017 & 0.15 & 0.27 & 0.065 & 0.273 & 2 & 1 & 0.1 & 0.05 & 0.01 & 20 \\ 
\hline
\end{tabular}

\label{tab:parameters}
\end{table}

\begin{figure}[htbp!]
\centering
\includegraphics[scale=0.45]{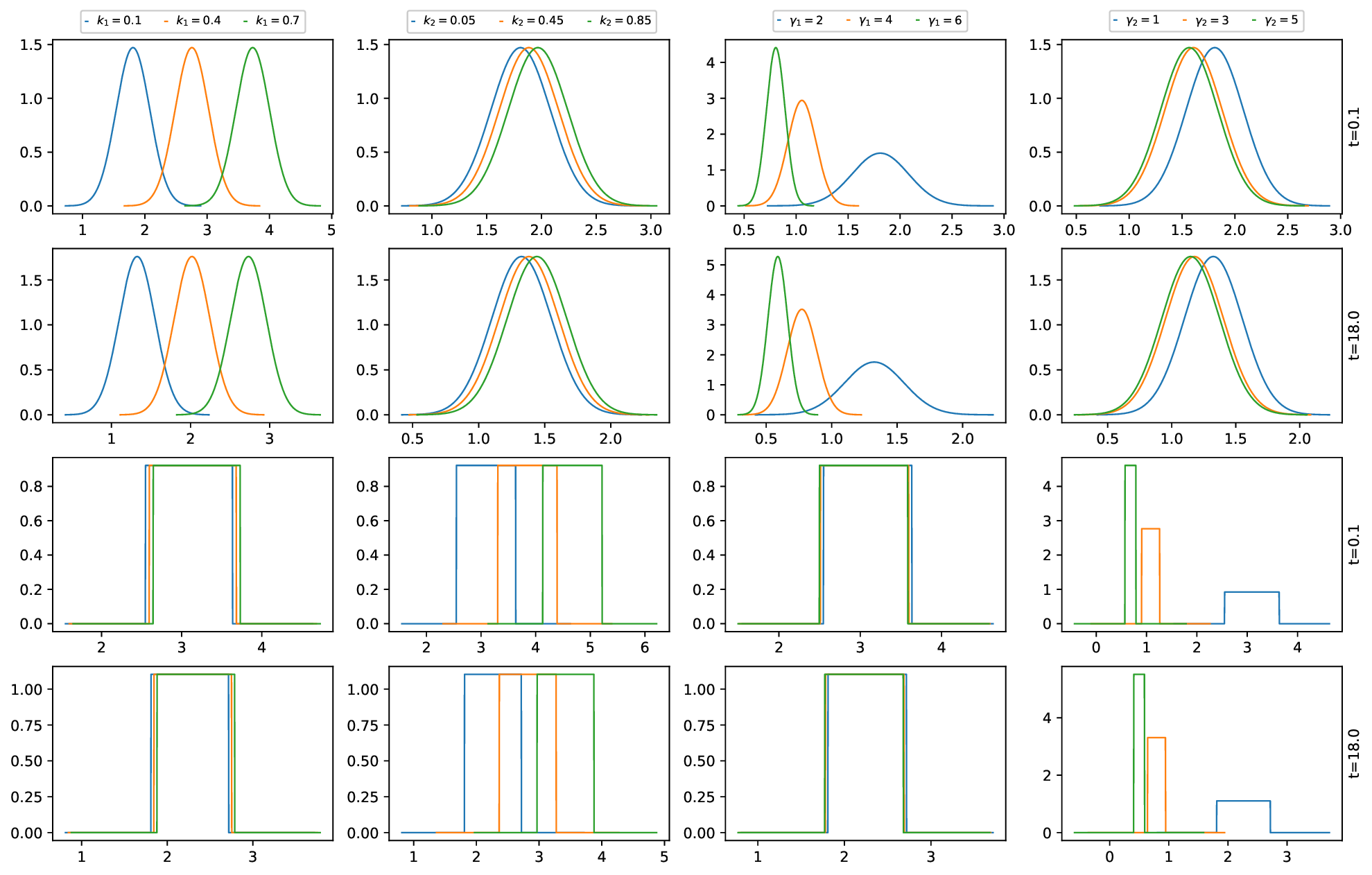}
\caption{The effects of $t$,  $k_1$, $k_2$, $\gamma_1$, and $\gamma_2$  on the  Nash equilibrium}
\label{fig:1}
\end{figure}
 In Figure \ref{fig:1}, the first and second rows display the density functions for Agent 1 at $t = 0.1$ and $t = 18$, respectively, while the third and fourth rows correspond to Agent 2’s density functions. For clarity, we focus on the characteristics of the parameters for Agent 1, as Agent 2’s performance with respect to these parameters is similar. The key observations are as follows.
\begin{itemize}
\item[(i)] As $k_1$ increases, Agent 1 tends to adopt riskier strategies, leading to a higher mean investment in risky assets. This suggests that greater sensitivity to the opponent's performance enhances Agent 1’s motivation to outperform.
\item[(ii)]  As Agent $1$'s  risk aversion parameter $\gamma_1$  rises, the mean of the equilibrium distribution decreases. This indicates that higher levels of risk aversion prompt Agent 1 to adopt more cautious strategies, reducing their expected investment in risky assets.
\item[(iii)] As Agent $2$ becomes more sensitive  to Agent $1$'s behavior  (i.e.,  as $k_2$  increases), Agent 1 tends to adopt riskier strategies to increase the likelihood of achieving higher returns. This suggests that Agent 2's increased sensitivity further motivates Agent 1 to excel.
\item[(iv)]  As Agent $2$ becomes more risk-seeking (i.e., as $\gamma_2$ decreases), Agent 1 adopts more aggressive strategies, investing a higher mean in risky assets. This adjustment is necessary to prevent Agent 1 from losing market share or competitive advantages if they fail to align their strategy with the increased risks.

\item[(v)] Comparing the equilibrium strategies between the first and second rows, the main difference lies in time 
 $t$. Since a time-decaying temperature parameter $ \lambda_1(t)$ is employed, it means that as time progresses, the weight assigned to exploration  decreases as time progresses. Consequently, the variance of the equilibrium strategies also decreases, as reflected in \eqref{sigma}.
 \end{itemize}
 It is noteworthy that an agent's own parameters consistently have a significantly greater impact on their strategy than the parameters of the opponent.

Now we conduct numerical experiments with simulated data to demonstrate our  Algorithm \ref{alg:1}. We first emphasize that there are some factors that affect the accuracy of the algorithm. First of all, Theorems \ref{responceconvergence} and \ref{fixpoint} give theoretical convergence, which provides support for algorithm design, but the theoretical results depend on the parameters of the market model and these parameters are actually unknown. Secondly, as mentioned earlier, a very important property in classical reinforcement learning is that the strategy obtained after each update is better than before, but this property no longer holds in time-inconsistent problem. These increases the impact of errors on the convergence in each algorithm iteration, which may affect the convergence to the true equilibrium strategy. Finally, due to the use of maximizing $\sum\limits_{k=0}^{N-1}L_i(\Phi;t_k,\hat X_i(t_k),Y(t_k))$ instead of maximizing $L_i(\Phi;t,\hat x,y)$ for all possible $t,\hat x,y$, each iteration highly relies on the current sample, and different samples may cause the algorithm to converge to different strategies. Meanwhile, similar to classical reinforcement learning, due to the time horizon is finite, each sample only has $N+1$ time points, which is far less than the number required to make the strategy maximizing $\sum\limits_{k=0}^{N-1}L_i(\Phi;t_k,\hat X_i(t_k),Y(t_k))$ and maximizing $L_i(\Phi;t,\hat x,y)$ for all possible $t,\hat x,y$ close enough. In summary, due to various reasons, algorithms for time-inconsistency problems rely more on model settings, especially parameter selection, than the classical reinforcement learning.

We use the stock process parameters detailed in Table \ref{tab:parameters}. Additionally, the other parameter settings for the algorithm are presented in Table \ref{tab:parameters2}.  \begin{table}[h!]
\centering
\caption{Parameter settings for the algorithm}
\label{tab:algorithm_parameters}
\begin{tabular}{|c|c|c|c|c|c|c|c|c|}
\hline
$T$ & $N$ & $\alpha$ & $\gamma_1$ & $\gamma_2$ & $k_1$ & $k_2$ & $\lambda_1$ & $\lambda_2$ \\ \hline
1   & 250 & 0.001    & 2          & 3          & 0.1   & 0.05  & 0.015       & 0.02        \\ \hline
\end{tabular}\label{tab:parameters2}
\end{table}
 For the hyper-parameters in the Adam algorithm, we use the default values given in \cite{KB14}. The Critic is parameterized as in \eqref{paravg} with $p(\theta,t)$  chosen as
\begin{align*}
	p(\theta,t)=\theta_1t^2+\theta_0t.
\end{align*}
The Actor is parameterized as \eqref{paraq} with suitable initial values selected based on the problem context. 
We use $M=50,000$ samples to train the model. As mentioned earlier, since the sample directly affects the results, we conduct  10 experiments and take the average value as our final result. Finally, we plot the mean value of the discounted value invested in the risk asset when the market state $Y(t)=0.273$ for Nash equilibrium in Figure \ref{fig:2}. 
\begin{figure}[htbp!]
\centering
\includegraphics[scale=0.45]{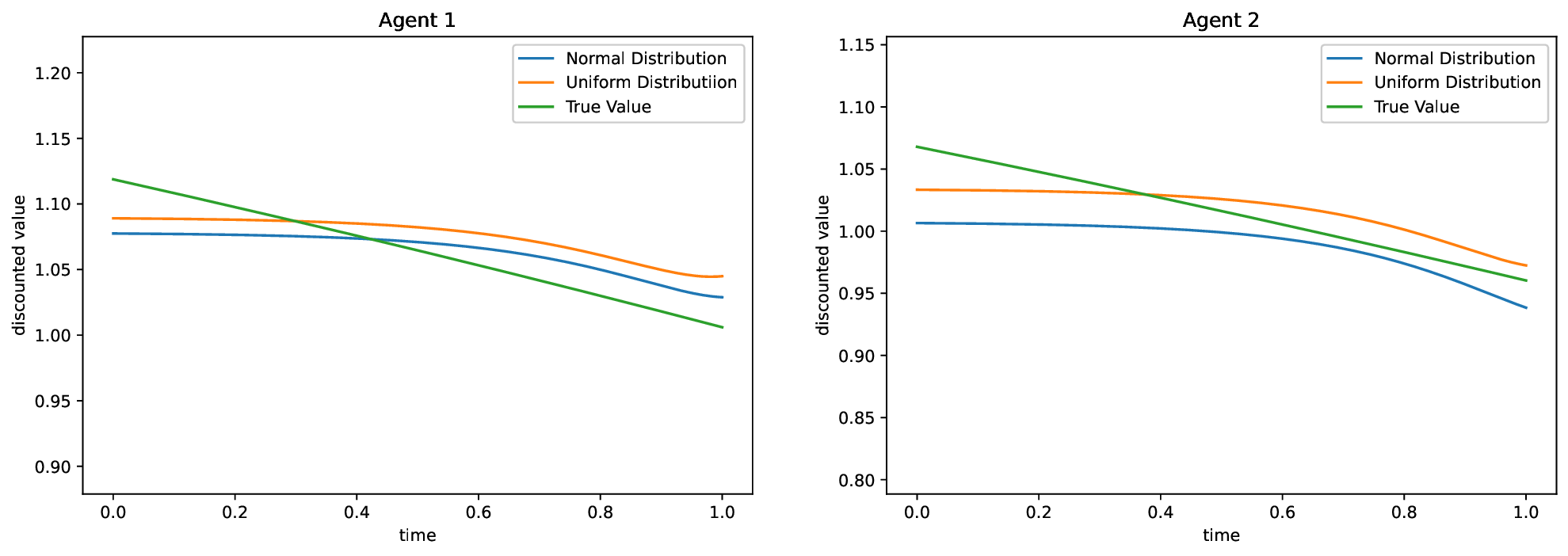}
\caption{The mean value of Nash equilibrium}
\label{fig:2}
\end{figure}

We examine the performance of the mean value of Nash equilibrium when both agents follow normal or uniform distributions. In particular, the forms of the Choquet regularizers can be found in \eqref{Normal} and \eqref{Gini}.  By comparing the learned policy with the true policy, we observed that our experimental results closely approximate the theoretical values under proper initial values. This underscores the effectiveness of our approach in approximating Nash equilibrium solutions.

\section{Conclusion}\label{sec:6}
This paper introduces a framework for continuous-time RL  in a competitive market, where two agents consider both their own wealth and their opponent's wealth under the mean-variance criterion. The Nash equilibrium distributions are derived  through dynamic programming, revealing that an agent's mean of equilibrium distribution for exploration is influenced not only by his own parameters but also by those of  his  opponent, while the variance of the distribution is solely determined by   the agent's  own model parameters.

In addition, we investigate the impact of model parameters on the equilibrium strategies, providing valuable insights into decision-making dynamics in competitive financial markets. Furthermore, we design an algorithm to study Nash equilibrium policies, and our experimental results indicate that the output of our algorithm closely approximates the theoretical value. Since the performance of our algorithm depends on the initial value selection, exploring better algorithms to learn the initial value setting and ensuring algorithm stability over time would be critical directions for future research.

\vspace{0.7cm}
\noindent\textbf{Acknowledgements.}~~The  research of  Junyi Guo is supported by the National Natural Science Foundation of China (No. 11931018 and 12271274). The  research of  Xia Han  is supported by the National Natural Science Foundation of China (Grant No. 12301604, 12371471) and the Fundamental Research Funds for the Central Universities, Nankai University (Grant No. 63231138).    


\end{document}